\documentclass[12pt,reqno]{amsart}
\usepackage{fullpage}
\usepackage{times}
\usepackage{amsmath,amssymb,amsthm,url}
\usepackage[utf8]{inputenc}
\usepackage[english]{babel}
\usepackage{comment}
\usepackage{bbm}
\usepackage{enumerate}

\usepackage{bm}
\usepackage{graphicx}
\usepackage{mathrsfs}
\usepackage[colorlinks=true, pdfstartview=FitH, linkcolor=blue, citecolor=blue, urlcolor=blue]{hyperref}
\usepackage[vlined,ruled]{algorithm2e}
\usepackage{mathtools}
\usepackage{cancel}
\newtheorem{thm}{Theorem}[section]
\newtheorem{lem}[thm]{Lemma}
\newtheorem{prop}[thm]{Proposition}

\theoremstyle{definition}
\newtheorem{notn}[thm]{Notation}
\newtheorem{defn}[thm]{Definition}

\newtheorem{rem}[thm]{Remark}
\newtheorem{ex}[thm]{Example}
\newcommand{\cF}{{\mathcal F}}
\newcommand{\Res}{\operatorname{Res}}
\newcommand{\ep}{\varepsilon}

\newcommand{\N}{\mathbb{N}}
\newcommand{\Z}{\mathbb{Z}}

\renewcommand{\pmod}[1]{{\ifmmode\text{\rm\ (mod~$#1$)}\else\discretionary{}{}{\hbox{ }}\rm(mod~$#1$)\fi}}

\newcommand{\faf}{{\mathfrak a}_f}
\newcommand{\fbf}{{\mathfrak b}_f}

\begin{document}

\title{Oscillation results for the summatory functions of fake mu's}
\author{Greg Martin}
\address{Department of Mathematics \\ University of British Columbia\\ Vancouver  V6T 1Z2 \\ Canada}
\email{gerg@math.ubc.ca}
\author{Chi Hoi Yip}
\address{School of Mathematics\\ Georgia Institute of Technology\\ Atlanta, GA 30332\\ United States}
\email{cyip30@gatech.edu}


\begin{abstract}
Mossinghoff, Trudgian, and the first author~\cite{MMT23} recently introduced a family of arithmetic functions called ``fake~$\mu$'s'', which are multiplicative functions for which there is a $\{-1,0,1\}$-valued sequence $(\ep_j)_{j=1}^{\infty}$ such that $f(p^j) = \ep_j$ for all primes~$p$. They investigated comparative number-theoretic results for fake~$\mu$'s and in particular proved oscillation results at scale $\sqrt{x}$ for the summatory functions of fake~$\mu$'s with $\ep_1=-1$ and $\ep_2=1$. In this paper, we establish new oscillation results for the summatory functions of all nontrivial fake~$\mu$'s 
at scales $x^{1/2\ell}$ where~$\ell$ is a positive integer (the ``critical index'') depending on~$f$; for $\ell=1$ this recovers the oscillation results in~\cite{MMT23}.
Our work also recovers results on the indicator functions of powerfree and powerfull numbers; we generalize techniques applied to each of these examples to extend to all fake $\mu$'s.
\end{abstract}

\maketitle

\section{Introduction}

A major topic in comparative prime number theory is the behavior of summatory functions of various multiplicative functions. For example, let $\mu(n)$ be the M\"obius function, and let $M(x) = \sum_{n \leq x} \mu(n)$ be its summatory function. In 1897, Mertens~\cite{1897.Mertens} conjectured that $|M(x)|\leq \sqrt{x}$ for all $x\geq1$, an assertion that subsequently became known as the \emph{Mertens conjecture}.
Similarly, let $\lambda(n) =(-1)^{\Omega(n)}$ be the Liouville function, and let $L(x) = \sum_{n \leq x}\lambda(n)$ be its summatory function. In 1919, P\'olya~\cite{1919.Polya} asked whether $L(x) \le 0$ holds for all $x$, a question that became known as the \emph{P\'olya problem} (often mistakenly called ``P\'olya's conjecture'').

One of the motivations for studying these problems was the fact that a positive answer to either would imply the Riemann hypothesis (RH), as P\'olya noted for $L(x)$ in~\cite{1919.Polya}. Indeed, by 1942 it was ``well known'', as reported by Ingham~\cite{I42}, that both RH and the simplicity of all zeros of the Riemann zeta function $\zeta(s)$ would follow from either of $M(x)/\sqrt x$ or $L(x)/\sqrt x$ being bounded either above or below by an absolute constant. However, Ingham showed that any of these one-sided bounds would also imply that there were infinitely many linear relations with integer coefficients among the positive imaginary parts of the zeros of $\zeta(s)$, which cast doubt upon both the Mertens conjecture and a positive answer to the P\'olya problem.

The P\'olya problem was first resolved in the negative by Haselgrove~\cite{H58}, who showed that $L(x)$ changes sign infinitely often.
Similarly, the Mertens conjecture was first disproved by Odlyzko and te Riele~\cite{OtR85}. Currently the best known lower bound on $\limsup_{x\to \infty} L(x)/\sqrt{x}$ and the best known upper bound on $\liminf_{x\to \infty} L(x)/\sqrt{x}$ are due to Mossinghoff and Trudgian~\cite{MT17}, and the corresponding best known bounds for $M(x)/\sqrt{x}$ are due to Hurst~\cite{H18}. We still do not know how to disprove either $L(x) \ll \sqrt{x}$ or $M(x) \ll \sqrt{x}$ or even any of the corresponding one-sided bounds.
The Mertens conjecture and the P\'olya problem motivated substantial work in comparative prime number theory; we refer the reader to an annotated bibliography for comparative prime number theory~\cite{ABCPNT} for further sources.

Recently, Mossinghoff, Trudgian, and the first author~\cite{MMT23} developed comparative number-theoretic results for a family of arithmetic functions they called ``fake~$\mu$'s'', which are multiplicative functions~$f$ such that for each positive integer~$j$, there is a constant $\ep_j \in \{-1,0,1\}$ such that $f(p^j) = \ep_j$ holds for all primes~$p$. We say~$f$ is \emph{defined via the sequence $(\ep_j)_{j=1}^{\infty}$}, and throughout this paper, we always identify~$f$ with its defining sequence $(\ep_j)_{j=1}^{\infty}$. Certainly fake~$\mu$'s include the M\"obius function (the real $\mu$!)~and the Liouville function. These authors focused on the ``persistent bias" and ``apparent bias", at the scale of~$\sqrt{x}$, of the summatory function of a fake~$\mu$. In particular, they showed~\cite[Theorem~3]{MMT23} that if~$f$ is a fake~$\mu$ with $\ep_1=-1$ and $\ep_2=1$, then its summatory function $F_f(x)$ satisfies
\begin{equation} \label{MMT}
F_f(x)-{\mathfrak a}\sqrt{x}= \Omega_{\pm} (\sqrt{x}),
\end{equation}
where the apparent bias~${\mathfrak a}$ is twice the residue at~$s= \frac12$ of the meromorphic function defined by the Dirichlet series $\sum_{n=1}^\infty f(n)n^{-s}$.
In other words, $F_f(x)-{\mathfrak a}\sqrt{x}$ is infinitely often larger than some positive constant times~$\sqrt x$ and also infinitely often smaller than some negative constant times~$\sqrt x$ (these constants can depend on~$f$). There is a small gap in their proof, but it can be filled with a few additional observations; see Example~\ref{rem:gap}. 

These authors included the comment that ``a function with no bias at scale~$\sqrt{x}$ could well see one at a smaller scale". This remark motivates the study of the current paper, namely, to establish new oscillation results for the summatory functions of a larger family of fake~$\mu$'s, with oscillations at potentially smaller scales than~$\sqrt x$. Indeed, we will unconditionally establish such an oscillation result for the summatory function of every nontrivial fake~$\mu$ (see Theorems~\ref{thm:weak}, \ref{thm:mainmu}, \ref{thm:mainkfree}, and~\ref{thm:main_full} below). We also establish upper bounds, both unconditional and assuming RH, on the error terms in the asymptotic formulas for all these summatory functions (see Theorems~\ref{thm:unconditional} and~\ref{thm:ubRH} below).

\subsection{Existing examples of fake~$\mu$'s}\label{sec1.1}

Before introducing our main results formally, we first describe a few subfamilies of fake~$\mu$'s whose summatory functions have been studied extensively.
(We refer the reader to the survey \cite[Chapter VI]{SMC06} for more results related to these fake~$\mu$'s.)
These examples also motivate us to divide fake~$\mu$'s into three categories---see Definition~\ref{defn:3type}.

\begin{ex}[Tanaka's M\"obius functions] \label{ex:tanaka}
Given an integer $k \geq 2$, recall that an integer is called {\em~$k$-free} if it is not divisible by the~$k$th power of any prime (for $k=2$ and $k=3$, these numbers are commonly called {\em squarefree} and {\em cubefree}, respectively).
Tanaka~\cite{T80} defined the generalized M\"obius function $\mu_k$ by declaring that $\mu_k(n) =(-1)^{\Omega(n)}$ if~$n$ is~$k$-free and $\mu_k(n) =0$ otherwise. Note that these functions interpolate between the M\"obius and Liouville functions in the sense that 
$\mu_2= \mu$ and $\lim_{k\to\infty} \mu_k= \lambda$ as a pointwise limit.
\begin{itemize}
\item We see that~$\mu_k$ is the fake~$\mu$ corresponding to the sequence $(\ep_j)$ defined by $\ep_j = (-1)^j$ for $1\le j \le k-1$ and $\ep_j=0$ for $j \geq k$.
\item The corresponding Dirichlet series $\sum_{n=1}^\infty \mu_k(n)n^{-s}$ can be written down explicitly and admits a nice factorization in terms of the Riemann zeta function. For example, when $k\geq 3$ is odd, we have the identity $\sum_{n=1}^\infty \mu_k(n)n^{-s} = \zeta(2s)\zeta(ks)/\zeta(s)\zeta(2ks)$.
\item Let $M_k(x) = \sum_{n \leq x} \mu_k(n)$. Tanaka~\cite{T80} showed that $M_k(x)-\tau_k\sqrt{x}= \Omega_{\pm}(\sqrt{x})$, where $\tau_2=0$, $\tau_k= \zeta(\frac{k}{2}) / \zeta(\frac{1}{2}) \zeta(k)$ if $k\ge3$ is odd, $\tau_k=1 / \zeta(\frac{1}{2}) \zeta(\frac{k}{2})$ if $k\ge4$ is even, and $\tau_{\infty}=1/\zeta(\frac12)$. 
These statements are special cases of the result~\eqref{MMT} that was proved later.
\end{itemize}
\end{ex}

\begin{ex}[Indicator functions of~$k$-free numbers] \label{ex:kfree}
Note that $\mu_k^2(n)$ is the indicator function of~$k$-free numbers (generalizing the fact that $\mu^2(n)$ is the indicator function of squarefree numbers).
\begin{itemize}
\item We see that~$\mu_k^2$ is the fake~$\mu$ corresponding to the sequence defined by $\ep_1= \cdots= \ep_{k-1}=1$ and $\ep_j=0$ for $j \geq k$.
\item The corresponding Dirichlet series also factors nicely in terms of $\zeta(s)$: we have the identity $\sum_{n=1}^\infty \mu_k^2(n)n^{-s} = \zeta(s)/\zeta(ks)$.
\item Let $Q_k(x)$ be the number of~$k$-free numbers up to~$x$, and let $R_k(x) = Q_k(x)-x/\zeta(k)$. Montgomery and Vaughan~\cite{MV81} showed under RH that $R_k(x) \ll_\ep x^{1/(k+1)+\ep}$. (This result has been slightly improved by various authors; for~$k$ sufficiently large, the best known bound is due to Graham and Pintz~\cite{GP89}.) Meng~\cite{M17} gave, under the additional assumption $\sum_{0 <\gamma \leq T}|\zeta'(\rho)|^2 \ll_{\ep} T^{1+\ep}$ for all $\ep>0$, a bound on an integral involving $R_k(x)$ which implies that $R_k(x) \ll x^{1/2k}$ on average. 
\item On the other hand, Evelyn and Linfoot~\cite{EL31} first proved that $R_k(x) = \Omega_{\pm}(x^{1/2k})$. The recent paper~\cite{MOT21} by Mossinghoff, Oliveira e~Silva, and Trudgian provides the best known explicit lower bounds on the oscillations of the error term; see also the survey~\cite{P05} by Pappalardi.
\end{itemize}
\end{ex}

\begin{ex}[Apostol's M\"obius functions]\label{ex:apostol}
Apostol~\cite{A70} described a different generalization of the M\"obius function. For each $k \geq 1$, let $\nu_k$ be the multiplicative function such that for each prime~$p$, we set $\nu_k(p^j) =1$ if $j<k$, $\nu_k(p^j) =-1$ if $j=k$, and $\nu_k(p^j) =0$ if $j>k$. Note that $\nu_1 = \mu$.
\begin{itemize}
\item We see that~$\nu_k$ is the fake~$\mu$ corresponding to the sequence defined by $\ep_1= \cdots= \ep_{k-1}=1$, $\ep_k=-1$, and $\ep_j=0$ for $j \geq k+1$.
\item While the corresponding Dirichlet series does not admit (for $k\ge2$) an exact factorization in terms of $\zeta(s)$, it does possess~\cite[Lemma~2.5]{BFMT23} a useful partial factorization of the form $\zeta(s)/ \zeta(ks)^2 \cdot A_k^*(s)$, where $A_k^*(s)$ is an Euler product that is absolutely convergent (and thus analytic) for $\Re (s) > 1/(k+1)$.
\item For $k \geq 2$, Apostol~\cite{A70} showed that there is a constant $\phi_k$ such that $\sum_{n \leq x} \nu_k(x) = \phi_kx+O(x^{1/k}\log x)$; the factor multiplying $x^{1/k}$ in the error term has recently been improved by Banerjee {\it et~al.}~\cite{BFMT23}. Under RH, Suryanarayana~\cite{S77} improved the error term to $O \bigl( x^{4k/(4k^2+1)} \exp(C\frac{\log x}{\log \log x}) \bigr)$ for some positive constant~$C$.
\end{itemize}
\end{ex}

\begin{ex}[Indicator functions of~$k$-full numbers]\label{ex:kfull}
Given an integer $k\ge2$, recall that an integer is called {\em $k$-full} if every prime that divides it does so with multiplicity at least~$k$ (for $k=2$ these numbers are commonly called {\em powerfull} or {\em squarefull} numbers).
\begin{itemize}
\item We see that the indicator function of~$k$-full numbers is the fake~$\mu$ corresponding to the sequence defined by $\ep_1= \cdots= \ep_{k-1}=0$ and $\ep_j=1$ for $j \geq k$.
\item When $k=2$, the corresponding Dirichlet series is exactly $\zeta(2s)\zeta(3s)/\zeta(6s)$. When $k \geq 3$, the corresponding Dirichlet series admits~\cite[Proposition 1]{BRS88} the useful partial factorization
\[
\biggl( \prod_{j=k}^{2k-1} \zeta(js) \bigg/ \prod_{j=2k+2}^{4k+3} \zeta(js)^{a_j} \biggr) V(s),
\]
where the~$a_j$ are particular integers and where $V(s)$ is an Euler product that is absolutely convergent (and thus analytic) for $\Re (s) > 1/(4k+4)$.
\item Let $N_k(x)$ be the number of~$k$-full numbers up to $x$; then $N_k(x)$ admits an asymptotic formula of the form $N_k(x) = \sum_{k \leq j \leq 2k-1} \faf(j)x^{1/j}+\Delta^{N_k}(x)$. Various upper bounds on $\Delta^{N_k}(x)$ can be found in~\cite{M97, L16} and the references therein; in particular, under the Lindel\"of hypothesis, Ivi\'c~\cite{I78} showed that $\Delta^{N_k}(x) \ll_\ep x^{1/2k+\ep}$.
\item Bateman and Grosswald~\cite{BG58} showed that if $\rho$ is any zero of the Riemann zeta function such that $\zeta(\frac\rho2) \ne 0$ and
$\zeta(\frac\rho3) \ne 0$, then $\Delta^{N_2}(x) = \Omega_{\pm}(x^{\Re(\rho)/6})$. Thus in particular, by taking $\rho = \frac12\pm i\cdot14.1347\ldots$ to be a zero of the zeta function closest to the real axis, their result implies that $\Delta^{N_2}(x) = \Omega_{\pm}(x^{1/12})$. Balasubramanian, Ramachandra, and Subbarao~\cite{BRS88} showed that $\Delta^{N_2}(x) = \Omega(x^{1/10})$ and that $\Delta^{N_k}(x) = \Omega(x^{1/(2k+\sqrt{8k}+3)})$ for $k \geq 3$; more precisely, they showed that $\Delta^{N_k}(x) = \Omega(x^{1/2(k+r)})$ where~$r$ is the smallest positive integer such that $r(r-1)\geq 2k$.
\end{itemize}
\end{ex}

One commonality of the above examples is the idea of factoring out powers of $\zeta(s)$, $\zeta(2s)$, and so on from a Dirichlet series, either resulting in a complete factorization or else leaving a remaining factor with nicer analytic properties (a larger half-plane of absolute convergence, for example). This theme is present in many guises in analytic number theory. For example, if $f(n)$ is a multiplicative function such that~$\kappa$ is the average value of $f(p)$ over primes~$p$, the Selberg--Delange method (see for example~\cite[Chapter~II.5]{T15}) finds an asymptotic formula for the summatory function of $f(n)$ by factoring $\zeta(s)^\kappa$ out of the corresponding Dirichlet series, so that the leftover factor is typically analytic in a neighborhood of $s=1$.

If in fact $f(p) = \kappa$ exactly for all primes~$p$, then the resulting factorization is $\sum_{n=1}^\infty f(n)n^{-s} = \zeta(s)^\kappa U_1(s)$ where $U_1(s)$ is a Dirichlet series whose coefficients are supported on squarefull numbers. If $f(p^2)$ is also independent of~$p$, one can further write $\sum_{n=1}^\infty f(n)n^{-s} = \zeta(s)^\kappa \zeta(2s)^{\kappa'} U_2(s)$ for an appropriate constant~$\kappa'$ and a Dirichlet series $U_2(s)$ whose coefficients are supported on $3$-full numbers, and so on. These {\em partial zeta-factorizations} are already beneficial for analytic methods, and sometimes one can consider analogous {\em total zeta-factorizations} $\sum_{n=1}^\infty f(n)n^{-s} = \prod_{v=1}^{\infty} \zeta(vs)^{a_v}$. For example, Moree~\cite[Sections~2 and~3]{M00} and other authors have used these zeta-factorizations as a means of calculating certain number-theoretic constants to high precision; in another vein, Dahlquist~\cite[Section 2]{D52} recognized the importance of finite zeta-factorizations in the study of the natural boundary of analytic continuation for Dirichlet series (see also the later chapters of~\cite{dSW}). Moreover, when the exponents~$a_v$ are integers, then the resulting functions are meromorphic and we expect Perron's formula and contour integration to yield asymptotic formulas, and explicit formulas involving the zeros of $\zeta(s)$, for our summatory functions.

\smallskip
Returning now to the examination of fake~$\mu$'s, it turns out that partial zeta-factorizations of this type are important not only for the proofs, but even for the statements, of our oscillation results. In Section~\ref{sec:alg}, we will describe an algorithm for computing such zeta-factorizations that is designed specifically for the Dirichlet series of fake~$\mu$'s. We will write the result of such a zeta-factorization in the form
\begin{equation}\label{eq:index}
\sum_{n=1}^\infty f(n)n^{-s} = \prod_{j=1}^{\ell} \zeta(js)^{a_j} \cdot U_\ell(s),    
\end{equation}
where $a_1, \ldots, a_\ell$ are integers and $U_\ell(s)$ is of the form 
\begin{equation}\label{eq:Uell}
U_\ell(s) = \prod_{p} \biggl(1+\sum_{j= \ell+1}^{\infty} \frac{\eta_j}{p^{js}} \biggr)
\end{equation}
for certain constants~$\eta_j$ (so that the coefficients in the Dirichlet series for $U_{\ell}(s)$ are supported on $(\ell+1)$-full numbers).

We now proceed to define the terminology required to state our main results.

\subsection{Main results}

We quickly observe that not all fake~$\mu$'s exhibit oscillations in their partial sums. For instance, if~$f$ is the indicator function of $n=1$ (corresponding to the case $\ep_j\equiv0$), there is no oscillation result. Similarly, if~$f$ is the indicator function of~$k$th powers for some $k \geq 1$ (that is, if $\ep_{j}=1$ when $k \mid j$ and $\ep_j=0$ otherwise), then there is no oscillation result either beyond the trivial $\sum_{n \leq x} f(n) - x^{1/k} = \lfloor x^{1/k}\rfloor - x^{1/k} = \Omega_-(1)$. For this reason, we call the fake~$\mu$'s mentioned above \emph{trivial fake~$\mu$'s}. These observations lead to the following definition.

\begin{defn}
Let $\cF$ be the set of arithmetic functions consisting of all fake~$\mu$'s that are not trivial. In other words, $f\in\cF$ precisely when $f(n)$ is a multiplicative function such that:
\begin{enumerate}
\item there exists a $\{-1,0,1\}$-valued sequence $(\ep_j)_{j=1}^{\infty}$ such that $f(p^j) = \ep_j$ for every prime~$p$;
\item $f(n)$ is neither the indicator function of~$\{1\}$, nor the indicator function of the set of~$k$th powers for any $k\ge1$.
\end{enumerate}
For any $f\in\cF$, define $F_{f}(x) = \sum_{n \leq x} f(n)$ to be the summatory function of~$f$, and define $D_{f}(s) = \sum_{n=1}^{\infty} f(n)n^{-s}$ to be the Dirichlet series associated with~$f$.
\end{defn}

Our goal is to deduce an oscillation result for $F_{f}(x)$ based on analytic properties of $D_{f}(s)$, stated with the help of the indices introduced in the following definition.

\begin{defn}\label{defn:indices}
If~$f\in\cF$ is defined via the sequence $(\ep_j)$, we define the {\em initial index} of~$f$ to be the smallest number~$j$ such that $\ep_j\ne0$. We define the {\em critical index} of~$f$ to be the smallest number~$j$ for which a power of $\zeta(js)$ appears in the denominator of the zeta-factorization of $D_{f}(s)$. More precisely, if for $\sigma>1$ we can write $D_f(s)$ in the form given in equations~\eqref{eq:index} and~\eqref{eq:Uell}, then the critical index of~$f$ equals~$\ell$ precisely when $a_1, a_2, \ldots, a_{\ell-1}\ge0$ and $a_\ell<0$. 
\end{defn}

Given a zeta-factorization~\eqref{eq:index}, we expect (when $U_\ell$ is nicely behaved) that the right-hand side will have real poles at $s= \frac1j$ whenever $a_j>0$ (so that $\zeta(js)$ appears to some power in the numerator); we further expect that it will have complex poles with real parts equal to~$\frac1{2j}$ whenever $a_j<0$ (so that $\zeta(js)$ appears to some power in the denominator). When using Perron's formula and contour integration, the real poles are associated to the main term of the asymptotic formula, and the complex poles are associated to oscillatory terms. Consequently, we expect a main term for $F_f(x)$ of the form
\begin{equation}\label{eq:G}
G_{f}(x) = \sum_{j=1}^{2\ell} \Res\biggl(D_f(s) \frac{x^s}{s}, \frac{1}{j} \biggr),
\end{equation}
where $\Res(g(s),s_0)$ denotes the residue of $g(s)$ at $s=s_0$.
We will study oscillation results and upper bounds for the error term
\begin{equation}\label{eq:E}
E_{f}(x) =F_{f}(x)-G_{f}(x).
\end{equation}
Note that~$\ell$ here denotes the critical index of~$f$, and that the sum defining~$G_f$ has been deliberately taken up to exactly $2\ell$ for the following reason. As a function of~$x$, the residue at $s= \frac1j$ in equation~\eqref{eq:G} will have order of magnitude $x^{1/j}$, while the residues at the complex poles with real part equal to~$\frac1{2\ell}$ will oscillate with order of magnitude $x^{1/2\ell}$. Therefore the rightmost $2\ell$ potential real poles should be taken into account in the main term, but we expect that all subsequent real poles will give a negligible contribution compared to the oscillations of the error term.

\smallskip

We are now able to state the most general form of our main oscillation result, which holds for every nontrivial fake~$\mu$:

\begin{thm}\label{thm:weak}
Let $f \in \cF$. If~$\ell$ is the critical index of~$f$, then $E_f(x) = \Omega_{\pm} (x^{{1}/{2\ell}})$. In other words,
\[
F_f(x) = G_f(x) + E_f(x) = \sum_{j=1}^{2\ell} \Res\biggl(D_f(s) \frac{x^s}{s}, \frac{1}{j} \biggr) + \Omega_{\pm} (x^{{1}/{2\ell}}).
\]
\end{thm}

\noindent
(In this theorem and throughout this paper, all implicit constants in~$\Omega$ and~$O$-notation may depend upon~$f$.)

Given more information about the specific $f\in\cF$, of course, we should be able to be more specific about this main term and oscillation term. We would like to determine when the residues on the right-hand side equal~$0$ (as many of them will) and to write the nonzero residues more explicitly; we would like to increase the size of the oscillation term when possible (even if just by a logarithmic factor); and we would like to more explicitly determine what the critical index~$\ell$ actually is.
As it happens, we can already be quite a bit more specific simply by dividing the set of nontrivial fake~$\mu$'s into three types.

\begin{defn}\label{defn:3type}
Let $f\in\cF$ be a nontrivial fake~$\mu$.
\begin{enumerate}
\item We say that~$f$ is of {\em M\"obius-type} if the initial index~$k\ge1$ of~$f$ has the property that $\ep_k=-1$. The M\"obius function~$\mu$ and the Liouville function~$\lambda$ are certainly of M\"obius-type, as are Tanaka's M\"obius functions~$\mu_k$ from Example~\ref{ex:tanaka}. As we will see in Theorem~\ref{thm:mainmu}, the critical index also equals~$k$ in this case.
\item We say that~$f$ is of {\em powerfree-type} if $\ep_1=1$ (so that the initial index of~$f$ is~$1$). When $k\ge2$, the indicator functions~$\mu_k^2$ of~$k$-free numbers (see Example~\ref{ex:kfree}) are certainly of powerfree-type, as are Apostol's M\"obius functions~$\nu_k$ from Example~\ref{ex:apostol}. In this case, it is important to consider the smallest positive number~$k$ with $\ep_k \neq 1$, which is in some sense a ``measure of powerfreeness" (since this yields the correct value of~$k$ when~$f= \mu_k^2$, and also when $f= \nu_k$). As we will see in Theorem~\ref{thm:mainkfree}, the critical index equals this value of~$k$ in this case.
\item We say that~$f$ is of {\em powerfull-type} if the initial index of~$f$ is $k \geq 2$ and $\ep_k=1$.
When $k\ge2$, the indicator functions of~$k$-full numbers (see Example~\ref{ex:kfull}) are certainly of powerfull-type.
The initial index~$k$ is in some sense a ``measure of powerfullness" of~$f$.
Unlike the previous two cases, there is no simple formula for the critical index, although in Section~\ref{sec:alg} we give an algorithm for computing the critical index from the defining sequence $(\ep_j)$.
\end{enumerate}
\end{defn}

We will be able to be more concrete about the main terms for our summatory functions of fake~$\mu$'s with the following notation.

\begin{defn}\label{defn:doublepole}
For any $f\in\cF$ and any positive integer~$j$, define
\[
\faf(j) =j\Res(D_f(s),\tfrac{1}{j})
\quad\text{and}\quad
\fbf(j) =j^2\Res((s-\tfrac{1}{j})D_f(s),\tfrac{1}{j}).
\]
If $D_f(s)$ has at most a double pole at $s= \frac1j$, then the principal part of $D_f(s)$ is
\[
\frac{\fbf(j)}{j^2(s-1/j)^2} + \frac{\faf(j)}{j(s-1/j)},
\]
so that subtracting this expression from $D_f(s)$ results in a function that is analytic at $s= \frac1j$. Note that either or both of $\faf(j)$ and $\fbf(j)$ might equal~$0$.
\end{defn}

We may now state the following refinement of Theorem~\ref{thm:weak} for M\"obius-type fake~$\mu$'s.

\begin{thm}\label{thm:mainmu}
Let $f \in \cF$ be of M\"obius-type with initial index~$k$. Then the critical index of~$f$ also equals~$k$. Moreover, 
\[
G_f(x) = \sum_{\substack{k+1\leq j \leq 2k\\ \ep_j=1}} \faf(j)x^{1/j}
\quad\text{and}\quad
E_f(x) = \Omega_{\pm}(x^{{1}/{2k}}).
\]
\end{thm}

\begin{rem}
A particular case of this theorem is when $f \in \cF$ has $\ep_1=-1$ and $\ep_2=1$, in which case~$f$ is of M\"obius-type with initial index~$k=1$. We thus see that Theorem~\ref{thm:mainmu} generalizes~\cite[Theorem~3]{MMT23} as stated in equation~\eqref{MMT}, which itself includes the M\"obius function~$\mu$, the Liouville function~$\lambda$, and Tanaka's M\"obius functions~$\mu_k$ from Example~\ref{ex:tanaka} as special cases.
\end{rem}

\begin{rem} \label{happy modbius}
When we provide upper bounds for $E_f(x)$ in Theorem~\ref{thm:ubRH}, we will see
that assuming RH, this oscillation result $\Omega_{\pm}(x^{{1}/{2k}})$ for fake~$\mu$'s of M\"obius type is best possible up to factors of~$x^\ep$.
\end{rem}

\begin{rem} \label{tempting heuristic}
A tempting heuristic suggests that Theorem~\ref{thm:weak} and its refinements might always yield best-possible oscillation results: One can write $E_f(x) = F_f(x) - G_f(x)$ as a contour integral involving a meromorphic function whose rightmost singularities are the poles coming from the negative power of $\zeta(\ell s)$ in equation~\eqref{eq:index} (where~$\ell$ is the critical index of~$f$), since $G_f(x)$ is designed to cancel all the real poles of $D_f(s)$ with $\Re(s)>\frac1{2\ell}$. Assuming RH, these rightmost singularities are all on the line $\Re(s) = \frac1{2\ell}$, and contour integration might plausibly result in an explicit formula whose dominant terms have order of magnitude $x^{1/2\ell}$. However, estimating the contribution to $F_f(x)$ from the shifted contour is not straightforward, and indeed we know that this heuristic can fail in general---a counterexample is given by the indicator function of~$k$-full numbers (see Example~\ref{ex:kfullfactored} and Remark~\ref{interesting consequence} below for more details).
\end{rem}

We continue by stating the following refinement of Theorem~\ref{thm:weak} for powerfree-type fake~$\mu$'s.

\begin{thm}\label{thm:mainkfree}
Let $f \in \cF$ be of powerfree-type, and let~$k$ be the smallest positive integer such that $\ep_k \neq 1$. Then the critical index of~$f$ equals~$k$, and
\begin{multline*}
G_f(x) = \faf(1)x+\sum_{\substack{k+1\leq j \leq 2k-1\\ \ep_j>\ep_{j-1}}} \faf(j)x^{1/j}+
\sum_{\substack{k+1\leq j \leq 2k-1\\ \ep_j=1,\,\ep_{j-1}=-1}} \fbf(j) x^{1/j} \bigl( \tfrac{1}{j}\log x-1\bigr) \\
+ \begin{cases}
0, & \text{if }\ep_{2k}-\ep_{2k-1}+\ep_k\leq 0,\\
\faf(2k)x^{1/2k}, &\text{if } \ep_{2k}-\ep_{2k-1}+\ep_k=1,\\
\faf(2k)x^{1/2k}+\fbf(2k)x^{1/2k} \bigl( \frac{1}{2k}\log x-1\bigr), &\text{if } \ep_{2k}-\ep_{2k-1}+\ep_k=2.
\end{cases}
\end{multline*}
% \begin{cases}
% 0, & \text{if }\ep_{2k}-\ep_{2k-1}+\ep_k\leq 0,\\
% 2k\Res(D_f(s),\frac{1}{2k})x^{1/2k}, &\text{if } \ep_{2k}-\ep_{2k-1}+\ep_k=1,\\
% 2k\Res(D_f(s),\frac{1}{2k})x^{1/2k}\\
% \;{}+4k^2\Res((s-\frac{1}{2k})D_f(s),\frac{1}{2k}) x^{1/2k} \bigl( \frac{\log x}{2k}+1\bigr), &\text{if } \ep_{2k}-\ep_{2k-1}+\ep_k=2.
% \end{cases}
Moreover, $E_f(x) = \Omega_{\pm} (x^{1/2k}(\log x)^{|\ep_k|})$.
\end{thm}

\begin{rem}
Since the indicator function of~$k$-free numbers is a powerfree-type fake~$\mu$, we see that Theorem~\ref{thm:mainkfree} recovers the result of Evelyn and Linfoot~\cite{EL31} mentioned in Example~\ref{ex:kfree} on oscillations of the error term in the counting function for~$k$-free numbers. 
Theorem~\ref{thm:mainkfree} also provides the first oscillation result for
Apostol's M\"obius functions~$\nu_k$ from Example~\ref{ex:apostol}.
\end{rem}

For the third category of fake~$\mu$'s, namely those of powerfull-type, a more precise statement is much more complicated, in large part because even computing the critical index itself is not straightforward. In Theorem~\ref{thm:main_full} we will give a detailed version of our oscillation result for powerfull-type fake~$\mu$'s.

\smallskip

We complement the oscillation results described above with upper bounds on the error terms $E_f(x)$. Since we are partly motivated by trying to understand how strong those oscillation results are, we provide two such results: the first one is unconditional, and the second one assumes RH.

\begin{thm}\label{thm:unconditional}
Let $f\in \cF$. Unconditionally, we have the following upper bounds on $E_f(x)$:
\begin{enumerate}
    \item If~$f$ is of powerfull-type with initial index~$k$, then $E_f(x)\ll_\ep x^{1/(k+1)+\ep}$ for each $\ep>0$.
    \item If~$f$ is of M\"obius-type or powerfree-type with critical index~$k$, then 
    \begin{equation}\label{eq:boundpowerfree}
     E_f(x)\ll x^{1/k}\exp\biggl(-c\frac{(\log x)^{3/5}}{(\log \log x)^{1/5}} \biggr)
    \end{equation}
where~$c$ is some absolute positive constant.
\end{enumerate}
\end{thm}

\begin{rem}
Note that if $f\in\cF$ has initial index~$k$, then the first nonzero contribution to the main term~\eqref{eq:G} for $F_f(x)$ has order of magnitude $x^{1/k}$. Therefore these unconditional upper bounds for the error term are only of modest strength for powerfull-type fake~$\mu$'s, and when~$f$ is of M\"obius-type we do not even have power savings in~$x$. This challenge is already reflected in the classical cases $f= \mu$ and $f= \lambda$, where the upper bound~\eqref{eq:boundpowerfree} with $k=1$ is the best known estimate for the Mertens sum $M(x)$ and for the error term $\Delta^L(x) = L(x) - \sqrt x/\zeta(\frac12)$ in P\'olya's problem.
\end{rem}

\begin{thm}\label{thm:ubRH}
Let $f \in \cF$. Assuming the Riemann hypothesis, we have the following upper bounds on $E_f(x)$:
\begin{enumerate}
\item If~$f$ is of M\"obius-type or powerfull-type with initial index~$k$, then there exists a positive constant~$C$ such that $E_f(x)\ll x^{1/2k}\exp(C\frac{\log x}{\log \log x})$.
\item If~$f$ is of powerfree-type with critical index~$k$, then $E_f(x)\ll_\ep x^{1/(k+1)+\ep}$ for each $\ep>0$.
\end{enumerate}
\end{thm}

\begin{rem}\
\begin{enumerate}
\item When~$f$ is of M\"obius-type, Theorem~\ref{thm:ubRH}(a) implies (assuming RH) that the oscillations given in Theorem~\ref{thm:mainmu} are essentially best possible, as mentioned in Remark~\ref{happy modbius}.
\item When $f=\mu$, the best-known conditional upper bound on $M(x)$ is due to Soundararajan~\cite{S09}, where he showed that $M(x)\ll x^{1/2}\exp\left((\log x)^{1/2}(\log\log x)^{14}\right)$.
\item When~$f$ is the indicator function of $k$-full numbers, Theorem~\ref{thm:ubRH}(a) improves Ivi\'c's result mentioned in Example~\ref{ex:kfull}, although his result requires only the Lindel\"of hypothesis rather than the full RH.
\item There exist examples (see Example~\ref{ex:k+h} below) of powerfull-type fake~$\mu$'s with initial index~$k$ where the error-term oscillations are as large as $E_f(x) = \Omega_{\pm} (x^{1/(2k+2)})$; so Theorem~\ref{thm:ubRH}(a) is at least reasonably sharp for powerfull-type fake~$\mu$'s.
\item Theorem~\ref{thm:ubRH}(b) extends the result of Montgomery and Vaughan's result concerning the indicator function of~$k$-free numbers (see Example~\ref{ex:kfree}) to all powerfree-type fake~$\mu$'s with critical index $k$. In particular, Theorem~\ref{thm:ubRH}(b) applies when $f=\nu_k$ (see Example~\ref{ex:apostol}) and improves Suryanarayana's result \cite{S77} that $E_f(x)\ll x^{{4k}/(4k^2+1)+o(1)}$.
\end{enumerate}
\end{rem}

The examples in Section~\ref{sec1.1} are far from being a complete list of fake~$\mu$'s already studied in the literature. We introduce one additional family of fake~$\mu$'s to further illustrate Theorems~\ref{thm:unconditional} and~\ref{thm:ubRH}.

\begin{ex}
Bege~\cite{B01} introduced the following generalization of Apostol's M\"obius functions: given integers $2\leq k<m$, the function $\mu_{k,m}$ is the fake~$\mu$ defined via the sequence $(\ep_j)_{j=1}^{\infty}$ with
\[
\ep_j = \begin{cases}
1, &\text{if } 1\le j\le k-1, \\
-1, &\text{if } j=m, \\
0, &\text{otherwise.}
\end{cases}
\]
Note that $\mu_{k,m}$ is of powerfree-type with critical index $k$. Theorem~\ref{thm:unconditional}(b) recovers Bege's unconditional bound~\cite[Theorem~3.1]{B01}, while Theorem~\ref{thm:ubRH}(b) improves Bege's conditional bound $E_{\mu_{k,m}}(x)\ll x^{2/(2k+1)+o(1)}$ \cite[Theorem~3.2]{B01} under~RH to $E_{\mu_{k,m}}(x)\ll x^{1/(k+1)+o(1)}$.
\end{ex}

\subsection*{Notation}
We use several notational conventions that are standard in analytic number theory. In this paper,~$p$ always denotes a prime, and $\sum_p$ and $\prod_p$ represent sums and products over all primes. For a complex number~$s$ we write $s= \sigma+it$, so that $\sigma = \Re(s)$ and $t = \Im(s)$. In addition, $\rho= \beta+i\gamma$ denotes a general nontrivial zero of the Riemann zeta function $\zeta(s)$, so that $\beta = \Re(\rho)$ and $\gamma = \Im(\rho)$. We write $f(x) = \Omega(g(x))$ to mean $\limsup_{x\to\infty} |f(x)|/g(x)>0$, and $f(x) = \Omega_{\pm}(g(x))$ to mean both $\limsup_{x\to\infty} f(x)/g(x)>0$ and $\liminf_{x\to\infty} f(x)/g(x)<0$.

\subsection*{Outline of the paper}

In Section~\ref{sec:alg}, we provide an algorithm to compute the critical index of a function $f\in \cF$. In Section~\ref{sec:omgea}, we prove an oscillation result for $E_f(x)$ based on its critical index and a few other parameters from the algorithm (Theorem~\ref{thm:general} is the most general statement).
%In particular, we deduce Theorem~\ref{thm:weak}, Theorem~\ref{thm:mainmu}, and Theorem~\ref{thm:mainkfree}.
In particular, in Section~\ref{sec:applications} we complete the proofs of Theorems~\ref{thm:mainmu} and~\ref{thm:mainkfree}, which apply to M\"obius-type and powerfree-type fake~$\mu$'s, respectively, as well as giving a general result (Theorem~\ref{thm:main_full} below) for powerfull-type fake~$\mu$'s. Together these three results imply Theorem~\ref{thm:weak}.
Finally, in Section~\ref{sec:ub}, we study upper bounds on the error term $E_f(x)$ and prove Theorems~\ref{thm:unconditional} and~\ref{thm:ubRH}.

\section{Zeta-factorizations and the critical index}\label{sec:alg}

In this section, we provide some precise statements about the (partial) zeta-factorizations mentioned in the introduction, including closed formulas and bounds for the resulting exponents and coefficients. With these statements in place, we then describe an algorithm for computing zeta-factorizations with enough factors to determine the critical index of a fake~$\mu$ (recall Definition~\ref{defn:indices}). We also provide several zeta-factorization examples using this algorithm, which we compare to known results from the literature.

One viewpoint we wish to stress is that given a specific Euler product with known numerical coefficients, all analytic number theorists who produced a zeta-factorization of that Euler product would arrive at the same numerical answer using essentially the same procedure as one another. The difficulties lie not in the calculations themselves, but rather in finding an accessible notation we can use to record the results of zeta-factorizations in a general setting.

\subsection{Zeta-factorization of Dirichlet series} \label{zeta product section}
We start with the following lemma for ``one-step" zeta-factorization for a family of Dirichlet series relevant to our discussions. Later, we will apply the lemma recursively to obtain ``multi-step" zeta-factorizations.

\begin{lem}\label{lem:factor}
Let~$t$ be a positive integer, and let $(\eta_j)_{j=t}^{\infty}$ be a sequence of integers. Assume that the Euler product 
$$
A(s) = \prod_{p} \biggl( 1 + \sum_{j=t}^{\infty} \frac{\eta_j}{p^{js}} \biggr)
$$
converges absolutely for $\sigma>1$. Then for $\sigma>1$,
$$
A(s) = \zeta(ts)^{\eta_t} \cdot \prod_{p} \biggl( 1 + \sum_{j=t+1}^{\infty} \frac{\eta_j'}{p^{js}} \biggr),
$$
where the first several values of $\eta_j'$ are
\begin{equation*}
\eta_j'=
\begin{cases}
\eta_j, &\text{if } t+1\leq j \leq 2t-1, \\
\eta_{2t}-\frac12(\eta_t^2+\eta_t), &\text{if } j=2t.
\end{cases}
\end{equation*}
Moreover, $|\eta'_j|\leq (2j|\eta_t|)^{2|\eta_t|}\max_{n\leq j} |\eta_n|$ for all $j\ge t+1$.
\end{lem}

\begin{proof}
For convenience, we extend the sequence $(\eta_j)$ to all integers indices by defining $\eta_0=1$ and $\eta_j=0$ when $1\le j\le t-1$ or $j\le-1$.
Assume throughout that $\sigma>1$. Then
\begin{equation} \label{mult by zeta ts}
\zeta(ts)^{-\eta_t}A(s) = \prod_{p} \biggl( \sum_{j\in\Z} \frac{\eta_j}{p^{js}} \biggr) \biggl( 1-\frac{1}{p^{ts}} \biggr)^{\eta_t}.
\end{equation}
The lemma is trivial if $\eta_t=0$; we consider two cases according to the sign of~$\eta_t$.

If $\eta_t>0$, then equation~\eqref{mult by zeta ts} becomes
\begin{align*}
\zeta(ts)^{-\eta_t}A(s) &= \prod_{p} \biggl( \sum_{j\in\Z} \frac{\eta_j}{p^{js}} \biggr) \biggl( \sum_{m=0}^{\eta_t} \frac{(-1)^m \binom{\eta_t}{m}}{p^{m ts}} \biggr) 
= \prod_{p} \biggl( \sum_{k\in Z} \frac1{p^{ks}} \sum_{m=0}^{\eta_t} (-1)^m \binom{\eta_t}{m} \eta_{k-tm} \biggr).
\end{align*}
It follows that 
\begin{equation}\label{eq:case1}
A(s) = \zeta(ts)^{\eta_t} \cdot \prod_{p} \biggl( 1 + \sum_{j\in\Z} \frac{\eta_j'}{p^{js}} \biggr)
\quad\text{with}\quad
\eta_j'= \sum_{m=0}^{\eta_t} (-1)^m \binom{\eta_t}{m} \eta_{j-tm}.
\end{equation}
In particular,
\begin{equation*}
\eta_j'=
\begin{cases}
\eta_j=0, &\text{if } j\le -1 \text{ or } 1\leq j\leq t-1, \\
\eta_0=1, &\text{if } j=0, \\
\eta_t-\binom{\eta_t}1\eta_0 = 0, &\text{if } j=t, \\
\eta_j-\binom{\eta_t}1\eta_{j-t} = \eta_j, &\text{if } t+1\leq j \leq 2t-1, \\
\eta_{2t}-\binom{\eta_t}1\eta_t+\binom{\eta_t}{2}\eta_0 = \eta_{2t}-\frac12(\eta_t^2+\eta_t), &\text{if } j=2t.
\end{cases}
\end{equation*}
From equation~\eqref{eq:case1}, we also deduce that
$$
|\eta_j'|\leq \sum_{m=0}^{\eta_t} \binom{\eta_t}{m} |\eta_{j-tm}|\leq \max_{n\leq j} |\eta_n| \sum_{m=0}^{\eta_t} \binom{\eta_t}{m}\leq 2^{\eta_t} \max_{n\leq j} |\eta_n|,
$$
which establishes the lemma in the $\eta_t>0$ case.

On the other hand, if $\eta_t<0$, then equation~\eqref{mult by zeta ts} becomes
\begin{align*}
\zeta(ts)^{-\eta_t}A(s) &= \prod_{p} \biggl( \sum_{j\in\Z} \frac{\eta_j}{p^{js}} \biggr) \biggl( \sum_{m=0}^{\infty} \frac{\binom{m-\eta_t-1}{-\eta_t-1}}{p^{m ts}} \biggr) 
= \prod_{p} \biggl( \sum_{k\in Z} \frac1{p^{ks}} \sum_{m=0}^{\infty} \binom{m-\eta_t-1}{-\eta_t-1} \eta_{k-tm} \biggr).
\end{align*}
% $$
% \prod_{p} \biggl( \sum_{j=0}^{\infty} \frac{\eta_j'}{p^{js}} \biggr) = \prod_{p} \biggl( \sum_{j=0}^{\infty} \frac{\eta_j}{p^{js}} \biggr) \biggl( \sum_{j=0}^{\infty} \frac{1}{p^{tjs}} \biggr)^{-\eta_t}= \prod_{p} \biggl( \sum_{j=0}^{\infty} \frac{\eta_j}{p^{js}} \biggr) \biggl(
% \sum_{m=0}^{\infty} \frac{\binom{m-\eta_t-1}{-\eta_t-1}}{p^{m ts}} \biggr) 
% $$
and it follows that 
\begin{equation}\label{eq:case2}
A(s) = \zeta(ts)^{\eta_t} \cdot \prod_{p} \biggl( 1 + \sum_{j\in\Z} \frac{\eta_j'}{p^{js}} \biggr)
\quad\text{with}\quad
\eta_j'= \sum_{m=0}^{\infty} \binom{m-\eta_t-1}{-\eta_t-1} \eta_{j-tm}.
\end{equation}
holds for all $j\geq 0$. In particular, we have $\eta_0'=1$, and
\begin{equation*}
\eta_j'=
\begin{cases}
\eta_j=0, &\text{if } j\le -1 \text{ or } 1\leq j\leq t-1, \\
\eta_0=1, &\text{if } j=0, \\
\eta_t+\binom{-\eta_t}1\eta_0 = 0, &\text{if } j=t, \\
\eta_j+\binom{-\eta_t}1\eta_{j-t} = \eta_j, &\text{if } t+1\leq j \leq 2t-1, \\
\eta_{2t}+\binom{-\eta_t}1\eta_t+\binom{1-\eta_t}{2}\eta_0 = \eta_{2t}-\frac12(\eta_t^2+\eta_t), &\text{if } j=2t.
\end{cases}
\end{equation*}
From equation~\eqref{eq:case2}, we also deduce that
\begin{align*}
|\eta_j'| \leq \sum_{m=0}^{\lfloor j/t \rfloor} \binom{m-\eta_t-1}{-\eta_t-1} |\eta_{j-tm}| &\leq \max_{n\leq j} |\eta_n| \sum_{m=0}^{\lfloor j/t \rfloor} \binom{m+|\eta_t|-1}{|\eta_t|-1} \\
&= \max_{n\leq j} |\eta_n| \cdot \frac1{|\eta_t|} \biggl\lfloor \frac{j}{t}+1 \biggr\rfloor \binom{\lfloor j/t \rfloor+|\eta_t|}{|\eta_t|-1} \\
&\leq \max_{n\leq j} |\eta_n| \biggl( \frac{j}{t}+1 \biggr) \biggl( \frac{j}{t}+|\eta_t|\biggr)^{|\eta_t|} \leq \max_{n\leq j} |\eta_n| (2j|\eta_t|)^{2|\eta_t|},
\end{align*}
which establishes the lemma in the $\eta_t<0$ case.
\end{proof}

A follow-up lemma puts into context the significance of the coefficient bound at the end of Lemma~\ref{lem:factor}.

\begin{lem} \label{really does converge lemma}
Let~$n$ be a nonnegative integer. Suppose that there are positive constants $A,B$ such the Euler product
\begin{equation}
U_n(s) = \prod_{p} \biggl(1+\sum_{j=n+1}^{\infty} \frac{\eta_j}{p^{js}} \biggr)
\end{equation}
satisfies $|\eta_j|\leq (Aj)^B$ for all~$j$. Then $U_n(s)$ converges absolutely to an analytic function for $\sigma > \frac1{n+1}$.
\end{lem}

\begin{proof}
To show absolute convergence, we must bound
\begin{align*}
%\sum_{p} \biggl| \sum_{j=n+1}^{\infty} \frac{\eta_{j}}{p^{js}} \biggr| 
\sum_{p}  \sum_{j=n+1}^{\infty} \frac{|\eta_{j}|}{|p^{js}|}
\le \sum_{p} \sum_{j=n+1}^{\infty} \frac{(Aj)^B}{p^{j\sigma}} &= A^B \sum_p \frac1{p^{(n+1)\sigma}} \sum_{i=0}^\infty \frac{(i+n+1)^B}{p^{i\sigma}} \\
&\le A^B \sum_p \frac1{p^{(n+1)\sigma}} \sum_{i=0}^\infty \frac{(n+1)^B(i+1)\cdots(i+B)}{p^{i\sigma}} \\
%&= A^B(n+1)^BB! \sum_p \frac1{p^{(n+1)\sigma}} \sum_{i=0}^\infty \binom{i+B}B \frac1{p^{i\sigma}} \\
&= A^B \sum_p \frac{(n+1)^B}{p^{(n+1)\sigma}} B!(1-p^{-\sigma})^{-B-1} \ll_{A,B,n} \sum_p \frac1{p^{(n+1)\sigma}}
\end{align*}
which converges by the assumption $\sigma>\frac1{n+1}$. Moreover, this convergence is locally uniform in~$s$, which implies that the infinite product is indeed analytic.
\end{proof}

\subsection{An algorithm for computing the critical index of $f \in \cF$}

We begin by using Lemma~\ref{lem:factor} to quickly compute the desired factorization of Dirichlet series $D_f(s)$ for a M\"obius-type fake~$\mu$.

\begin{prop}\label{prop:typeI}
Let $f\in \cF$ be of M\"obius-type with initial index~$k$. Then the critical index of~$f$ also equals~$k$. Moreover, for $\sigma>1$, 
\begin{equation}
D_f(s) = U_{2k}(s) \cdot \prod_{j=k}^{2k} \zeta(js)^{\ep_j},    \quad \text{where} \quad U_{2k}(s) = \prod_{p} \biggl(1+\sum_{j=2k+1}^{\infty} \frac{\eta_j}{p^{js}} \biggr)
\end{equation}
with $|\eta_j|\leq (2j)^{2(k+1)}$.
\end{prop}

\begin{rem}
The coefficient bound $|\eta_j|\leq (2j)^{2(k+1)}$ implies, by Lemma~\ref{really does converge lemma}, that $U_{2k}(s)$ is analytic for $\sigma>1/(2k+1)$.
\end{rem}

\begin{proof}[Proof of Proposition~\ref{prop:typeI}]
We begin by setting $\theta_j^{(k)}=\ep_j$ for all~$j$ and writing
\begin{equation} \label{prop:typeI base case}
U_k(s) = \prod_p \biggl( 1 + \sum_{j=k}^\infty \frac{\theta_j^{(k)}}{p^{js}} \biggr) = D_f(s).
\end{equation}
We claim that for each $t=k,k+1,\dots,2k+1$, we can write
\[
D_f(s) = U_t(s) \prod_{j=k}^{t-1} \zeta(js)^{\ep_j}
\quad\text{with}\quad
U_t(s) = \prod_p \biggl( 1 + \sum_{j=t}^\infty \frac{\theta_j^{(t)}}{p^{js}} \biggr),
\]
where $\theta_j^{(t)} = \ep_j$ for all $t\le j\le 2k$ and
$|\theta_j^{(t)}| \le (2j)^{2(t-k)}$ for all~$j\in\N$. The base case $t=k$ is exactly equation~\eqref{prop:typeI base case}, whereas deriving the case $t+1$ from the case~$t$ is a direct application of Lemma~\ref{lem:factor}. (In the first step going from $t=k$ to $t=k+1$, it is important to note that $\ep_k=-1$ implies that $\ep_{2k} - \frac12(\ep_k^2+\ep_k)=\ep_{2k}$. The fact that $\ep_k=-1$ also confirms that the critical index of~$f$ equals~$k$.) At the end of this recursive process, the final case $t=2k+1$ is the statement of the proposition, with $\eta_j = \theta_j^{(2k+1)}$.
\end{proof}

Next, we consider $f \in \cF$ of powerfree-type and powerfull-type. In this case, before applying Lemma~\ref{lem:factor}, we need to first determine the critical index of~$f$. Algorithm~\ref{alg} below computes the critical index of~$f$, as well as principal indices of~$f$ defined below. Based on the algorithm, we further establish Theorem~\ref{thm:algorithm} on the partial zeta-factorization of $D_f(s)$ into the desired form~\eqref{eq:index}. The introduction of principal indices plays a crucial role in describing Algorithm~\ref{alg} as well as in stating Theorem~\ref{thm:algorithm}.

\begin{defn}\label{defn:principal indices}
Suppose that~$f\in\cF$ is defined via the sequence $(\ep_j)$ and has critical index~$\ell$, so that we can write $D_f(s)$ in the form given in equations~\eqref{eq:index} and~\eqref{eq:Uell} with $a_1, a_2, \ldots, a_{\ell-1}\ge0$ and $a_\ell<0$. We define the {\em principal~indices} of~$f$ to be those numbers $1\leq j\leq\ell-1$ for which $a_j>0$ (indicating that a power of $\zeta(js)$ is truly present in the numerator of the zeta-factorization~\eqref{eq:index}).
\end{defn}

In the algorithm below, for a positive integer~$j$ and a set of positive integers $\{c_1,c_2,\ldots, c_m\}$, we define the number of \emph{representations of~$j$ from $\{c_1,c_2,\ldots, c_m\}$}, denoted by~$n_j$ in the algorithm, to be the number of nonnegative integer solutions $(\alpha_1, \alpha_2, \ldots, \alpha_m)$ to the equation $\sum_{i=1}^m \alpha_i c_i=j$. 

\begin{algorithm} \label{alg}
$c_1 \gets \text{ initial index of~$f$}$\\
$m \gets 1$\\
$j \gets c_1+1$\\
\While {true}
{
$n_{j} \gets \text{the number of representations of } j \text { from } \{c_1,c_2,\ldots, c_m\}$\\
\If{$n_{j}=0$ \and $\ep_{j}=1$}
{
    $c_{m+1} \gets j$\\
    $m \gets m+1$\\
}
\If{$n_{j} > \ep_j$}
{
    $M \gets m$\\
    $\ell \gets j$\\
    \Return $\ell, c_1, c_2, \ldots, c_M$
}
$j \gets j+1$
}
\caption{Compute the critical index and principal indices of $f\in \cF$.}
\end{algorithm}

\begin{thm}\label{thm:algorithm}
Let $f\in \cF$ be of powerfree-type or powerfull-type. Algorithm~\ref{alg} terminates in finitely many steps and computes the critical index of~$f$ (denoted by $\ell$), as well as the principal indices $c_1<c_2<\cdots<c_M$ of~$f$. Moreover, for $\sigma>1$, we have the factorization
\begin{equation}\label{eq:factorkfull}
D_f(s) =U_{2\ell}(s) \cdot \frac{\prod_{j=1}^{M} \zeta(c_j s)}{\zeta(\ell s)^{n_{\ell}-\ep_{\ell}}} \cdot \prod_{j= \ell+1}^{2\ell} \zeta(js)^{a_j},
\end{equation}
with 
\begin{equation}\label{eq:a_j}
a_j=
\begin{dcases}
    \sum_{I \subset\{1,\ldots,M\}} (-1)^{\#I} \ep_{j-\sum_{i \in I}c_i}, & \text{if } \ell+1\leq j \leq 2\ell-1, \\
     -\frac{(\ep_{\ell}-{n_{\ell}})^2+\ep_{\ell}-{n_{\ell}}}2 + \sum_{I \subset\{1,\ldots,M\}} (-1)^{\#I} \ep_{2\ell-\sum_{i \in I}c_i}, & \text{if }j=2\ell,
\end{dcases}
\end{equation}
where $M<\ell$ and $n_{\ell}$ are defined in Algorithm~\ref{alg}, and we have set $\ep_0=1$ and $\ep_j=0$ for $j<0$.
Furthermore,
$$
U_{2\ell}(s) = \prod_{p} \biggl(1+\sum_{j=2\ell+1}^{\infty} \frac{\eta_j}{p^{js}} \biggr),
$$ 
and there exist constants~$A$ and~$B$, depending only on~$\ell$, such that $|\eta_j|\leq (Aj)^B$ for all~$j\ge2\ell+1$.
\end{thm}

\begin{rem}
Again, the coefficient bound $|\eta_j|\leq (Aj)^B$ implies that $U_{2\ell}(s)$ is analytic for $\sigma>1/(2\ell+1)$ by Lemma~\ref{really does converge lemma}.
\end{rem}

\begin{proof}[Proof of Theorem~\ref{thm:algorithm}]
We first show that Algorithm~\ref{alg} terminates in finitely many steps; equivalently, we show that the critical index of the sequence~$(\ep_j)$ is finite. If there is a positive integer~$j$ such that $\ep_{j}=-1$, then the second {\bf if} statement of the algorithm ensures that the critical index of~$(\ep_j)$ is at most~$j$. Thus, we can assume that $\ep_j \in \{0,1\}$ for all~$j$. Let~$k$ be the initial index of~$f$.
\begin{itemize}
\item If $\ep_j=0$ for all indices~$j$ that are not multiples of~$k$, then let~$k'$ be the smallest multiple of~$k$ such that $\ep_{k'}\neq 1$ (such a~$k'$ must exist by the exclusion of trivial fake~$\mu$'s from the family~$\cF$). Then the second {\bf if} statement of the algorithm ensures that the critical index of~$(\ep_j)$ is at most~$k'$.
\item Otherwise, let~$k'$ be the smallest integer with $\ep_{j}=1$ that is not a multiple of~$k$. Then the first {\bf if} statement of the algorithm sets $c_2=k'$. Since $\operatorname{lcm}(k,k')$ has at least two representations from $\{k,k'\} = \{c_1,c_2\}$, the second {\bf if} statement of the algorithm ensures that the critical index of~$(\ep_j)$ is at most~$\operatorname{lcm}(k,k')$.
\end{itemize}

Next we make the initial definitions $D_0(s) =D_f(s)$ and $\theta_j^{(0)}= \ep_j$ for all~$j$.

\smallskip{\em Main goal}: We will show inductively, for each $1\le m \leq M$, that~$c_m$ is the smallest positive integer~$j$ such that $\theta_j^{(m-1)}$ is nonzero, and moreover that $\theta_{c_m}^{(m-1)}=1$. These claim are trivial for $m=1$, as~$c_1$ is the initial index of~$f$, and $\ep_{c_1}=1$ since~$f$ is not of M\"obius-type. 

For the inductive step, fix $1\leq m \leq M-1$, and assume both that~$c_m$ is the smallest positive integer~$j$ such that $\theta_j^{(m-1)}$ is nonzero and that $\theta_{c_m}^{(m-1)}=1$. We write $\theta_0^{(m-1)}=1$, and for convenience we adopt the convention that $\theta_j^{(m-1)}=0$ when $j\le-1$. Set $D_m(s) = D_{m-1}(s)/\zeta(c_ms)$, and let $\theta_j^{(m)}$ be defined by
$$
D_m(s) = \prod_{p} \biggl( \sum_{j=0}^{\infty} \frac{\theta_j^{(m)}}{p^{js}} \biggr)
$$
for $\sigma>1$. By the definition of $D_m(s)$ we have $\theta_j^{(m)}= \theta_j^{(m-1)}-\theta_{j-c_m}^{(m-1)}$, and then by induction it is easy to show that
\begin{equation}\label{eq:epjm}
\theta_j^{(m)}= \sum_{I \subset \{1,\ldots,m\}} (-1)^{\#I} \theta^{(0)}_{j-\sum_{i \in I}c_i}
= \sum_{I \subset \{1,\ldots,m\}} (-1)^{\#I} \ep_{j-\sum_{i \in I}c_i}.
\end{equation}

We define three parameters $r_m,s_m,t_m$ as follows. Let~$r_m$ be the smallest positive integer with $\ep_{r_m}\neq 0$ such that~$r_m$ has no representations from $\{c_1,c_2,\ldots, c_m\}$. Let~$s_m$ be the smallest positive integer with $\ep_{s_m}\neq 1$ that has exactly one representation from $\{c_1,c_2,\ldots, c_m\}$. Finally, let~$t_m$ be the smallest positive integer that has at least two representations from $\{c_1,c_2,\ldots, c_m\}$. If any of these numbers $r_m,s_m,t_m$ does not exist, we regard it as $+\infty$.

Now set $c_{m+1}= \min \{r_m,s_m,t_m\}$; we claim that $c_{m+1}$ is finite. This is easy to see for $m \geq 2$, since in this case $c_1c_2$ has at least two representations from $\{c_1,c_2,\ldots, c_m\}$ and thus $c_{m+1}\leq t_m\leq c_1c_2$. It remains to consider the case $m=1$, in which every nonnegative integer automatically has at most one representation from~$\{c_1\}$. But $r_1=s_1=+\infty$ would mean that $\ep_j=1$ if~$j$ is a multiple of~$c_1$ and $\ep_j=0$ otherwise; however, this sequence results in a trivial fake~$\mu$ which has been ruled out in the definition of the family~$\cF$.

{\em Subgoal}: Next we consider the range $1\leq j\leq c_{m+1}$ and determine the values $\theta_{j}^{(m)}$ in this range. We will show that $\theta_{j}^{(m)}=0$ when $j<c_{m+1}$, and also that $\theta_{c_{m+1}}^{(m)}=\ep_{c_{m+1}} - n_{c_{m+1}} \ne0$. We will need to consider three different cases depending on which of $r_m,s_m,t_m$ is smallest. Note that~$c_{m+1}\le r_m$ and~$c_{m+1}\le s_m$ and~$c_{m+1}\le t_m$ by definition, so we will not need to consider values of~$j$ above any of these parameters.

\begin{enumerate}
\item Assume that~$j$ has no representations from $\{c_1,c_2,\ldots, c_m\}$. For each subset~$I$ of~$\{1,\dots,m\}$, it follows that $j-\sum_{i \in I}c_i$ has no representations from $\{c_1,c_2,\ldots, c_m\}$ either and thus $\ep_{j-\sum_{i \in I}c_i}=0$ by the definition of~$r_m$. Equation~\eqref{eq:epjm} then implies that $\theta_{j}^{(m)}= \ep_{j}$. By the definition of $r_m$, if $j<r_m$ then $\theta_{j}^{(m)}= \ep_{j}=0$, while if $j=r_m$ then $j=c_{m+1}$ and $\theta_{c_{m+1}}^{(m)}= \ep_{c_{m+1}} \in \{-1,1\}$.

\item Next, assume that~$j= \sum_{i=1}^m \alpha_i c_i$ has exactly one representation from $\{c_1,c_2,\ldots, c_m\}$. Let $X= \{1\leq i \leq m\colon \alpha_i>0\}$. Then for each subset~$I$ of~$X$, it follows that $j-\sum_{i \in J} c_i$ also has exactly one representation from $\{c_1,c_2,\ldots, c_m\}$ and thus $\ep_{j-\sum_{i \in J} c_i}=1$ by the definition of~$s_m$. On the other hand, if $I$ is a subset of $\{1,\dots,m\}$ such that $I \not \subset X$, then $j-\sum_{i \in I} c_i$ has no representations from $\{c_1,c_2,\ldots, c_m\}$, and thus $\ep_{j-\sum_{i \in I}c_i}=0$ by the definition of~$r_m$. Equation~\eqref{eq:epjm} and the binomial theorem then imply that
$$
\theta_{j}^{(m)}= \sum_{I \subset X} (-1)^{\#I} \ep_{j-\sum_{i \in I}c_i}= \ep_{j}+\sum_{b=1}^{\#X} (-1)^b \binom{\#X}{b}= \ep_{j}-1.
$$
By the definition of $s_m$, if $j<s_m$ then $\theta_{j}^{(m)}= \ep_{j}-1=0$, while if $j=s_m$ then $j=c_{m+1}$ and $\theta_{c_{m+1}}^{(m)}= \ep_{c_{m+1}}-1 \in \{-2, -1\}$.

\item Finally, assume that~$j$ has $n_j\ge2$ representations from $\{c_1,c_2,\ldots, c_m\}$; by the definition of~$t_m$, we must have $j=t_m=c_{m+1}$. Write these representations as $j=\sum_{i=1}^m \alpha^{(h)}_i c_i$ for $1\leq h \leq n_j$. Let $X_h= \{1\leq i \leq m\colon \alpha_i^{(h)}>0\}$; we claim that the~$X_h$ are pairwise disjoint. Indeed, if we had $\alpha_i^{(h_1)}>0$ and $\alpha_i^{(h_2)}>0$ for some $1\le i\le m$ and $1\le h_1<h_2\le n_j$, then $j-c_i$ would have at least two representations from $\{c_1,c_2,\ldots, c_m\}$, violating the definition of~$t_m$.

Let~$I$ be a nonempty subset of $\{1,\dots,m\}$. Then $j-\sum_{i \in I} c_i$ has at most one representation from $\{c_1,c_2,\ldots, c_m\}$ by the definition of~$t_m$. If $I \subset X_{h}$ for some $1\leq h \leq n_j$, then $j-\sum_{i \in I} c_i>0$ does have a representation from $\{c_1,c_2,\ldots, c_m\}$, and thus $\ep_{j-\sum_{i \in I} c_i}=1$ by the definition of $s_m$ and the assumption that $t_m<s_m$. On the other hand, if $I \not\subset X_{h}$ for every $1\leq h \leq n_j$, then $j-\sum_{i \in I} c_i$ has no representations from $\{c_1,c_2,\ldots, c_m\}$ (for any such representation would induce an additional representation of~$j$ itself). Thus, it follows that $\ep_{j-\sum_{i \in I} c_i}=0$ by the definition of $r_m$ and the assumption that $t_m<r_m$.
Equation~\eqref{eq:epjm} and the binomial theorem then imply that
$$
\theta_{t_m}^{(m)}= \ep_{t_m}+\sum_{h=1}^{n_j} \sum_{\substack{I \subset X_{h}\\ I \neq \emptyset}} (-1)^{\#I} \ep_{t_m-\sum_{i \in I}c_i}= \ep_{t_m}+\sum_{h=1}^{n_j} \sum_{b=1}^{\#X_{h}} (-1)^b \binom{\#X_{h}}{b}= \ep_{t_m}-n_j\leq -1.
$$
\end{enumerate}
We have therefore achieved our subgoal of showing that $\theta_{j}^{(m)}=0$
for $j<c_{m+1}$ and that $\theta^{(m)}_{c_{m+1}}= \ep_{c_{m+1}}-n_{c_{m+1}}\ne0$. 
We now consider the sign of $\theta^{(m)}_{c_{m+1}}$:
\begin{itemize}
\item The only way that $\theta^{(m)}_{c_{m+1}}>0$ is when $\ep_{c_{m+1}}=1$ and $n_{c_{m+1}}=0$, and thus $\theta^{(m)}_{c_{m+1}}=1$. Therefore we have finished the proof for the induction step for $m+1$. At this point, the first {\bf if} statement of the algorithm appends $c_{m+1}$ to the list of principal indices, increases~$m$ by~$1$, and repeats the {\bf while} loop. 
\item On the other hand, $\theta^{(m)}_{c_{m+1}}<0$ means that $\ep_{c_{m+1}}<n_{c_{m+1}}$. In this event, the second {\bf if} statement sets $\ell=c_{m+1}$ and $m=M$ and terminates the algorithm.
\end{itemize}
These observations complete the verification of our main goal.
\smallskip

Note that for $\sigma>1$,
\begin{equation}\label{eq:TM1}
D_M(s) = \frac{D_f(s)}{\prod_{j=1}^{M} \zeta(c_j s)}= \prod_{p} \biggl( \sum_{j=0}^{\infty} \frac{\theta_j^{(M)}}{p^{js}} \biggr).    
\end{equation}
We have shown that $\theta_j^{(M)}=0$ for $1\leq j<\ell$ and that $\theta_\ell^{(M)}= \ep_\ell-n_\ell<0.$ Also, note that equation~\eqref{eq:epjm} implies that $|\theta_j^{(M)}|\leq 2^M$ for all~$j$. In particular, 
$$
\biggl|
\theta_{2\ell}^{(M)}-\frac{(\theta_{\ell}^{(M)})^2+\theta_{\ell}^{(M)}}{2} \biggr|\leq 2^M+\frac{4^M+2^M}{2}< 2^{2M+1}.
$$
Now we can apply Lemma~\ref{lem:factor} inductively $\ell+1$ times to $D_M(s)$ to obtain
\begin{equation}\label{eq:TM2}
D_M(s) = \prod_{j= \ell}^{2\ell-1} \zeta(js)^{\theta_j^{(M)}} \cdot \zeta(2\ell s)^{\theta_{2\ell}^{(M)}-\frac12((\theta_{\ell}^{(M)})^2+\theta_{\ell}^{(M)})} U_{2\ell}(s),
\end{equation}
where 
$$
U_{2\ell}(s) = \prod_{p} \biggl(1+\sum_{j=2\ell+1}^{\infty} \frac{\eta_j}{p^{js}} \biggr)
$$ 
with $|\eta_j|\leq 2^M (2j \cdot 2^{2M+1})^{4^{M+1} (\ell+1)}$. (Note that this bound can be made to depend upon~$\ell$ alone since $M<\ell$.) Combining equations~\eqref{eq:epjm}--\eqref{eq:TM2} establishes equation~\eqref{eq:factorkfull}, which completes the proof of the theorem.
\end{proof}

\subsection{Examples of applying Algorithm~\ref{alg}}

It will be illuminating to give several examples of fake~$\mu$'s where we see explicitly the zeta-factorization resulting from Algorithm~\ref{alg}. In some of these examples we will take note of the oscillation results implied by Theorem~\ref{thm:weak}, even though that theorem has not yet been proved; we assure the reader that these examples are merely for the purposes of illustration and will not be used when we prove Theorem~\ref{thm:weak} in Section~\ref{sec:omgea}.

First, we apply Theorem~\ref{thm:algorithm} to study the factorization of $D_f(s)$ for a powerfree-type fake~$\mu$.

\begin{prop}\label{prop:typeII}
Let $f \in \cF$ be of powerfree-type, and let~$k$ be the smallest positive integer such that $\ep_k \neq 1$.  Then the critical index of~$f$ equals~$k$. Moreover, for $\sigma>1$, we can write
$$
D_f(s) = \frac{\zeta(s)}{\zeta(ks)^{1+|\ep_k|}} \biggl( \prod_{j=k+1}^{2k-1} \zeta(js)^{\ep_j-\ep_{j-1}} \biggr) \zeta(2ks)^{\ep_{2k}-\ep_{2k-1}-|\ep_k|} U_{2k}(s),
$$    
where 
$$
U_{2k}(s) = \prod_{p} \biggl(1+\sum_{j=2k+1}^{\infty} \frac{\eta_j}{p^{js}} \biggr),
$$ 
has the property that there exist $A,B>0$ such that $|\eta_j|\leq (Aj)^{B}$ for all~$j$.  
\end{prop}
\begin{proof}
We follow the notation used in Algorithm~\ref{alg}. We have $c_1=1$ and thus $n_j \geq 1$ for all~$j$ since $j=1+1+\cdots+1$. Since $\ep_2= \cdots= \ep_{k-1}=1$ and $\ep_k<1=n_k$, we have $\ell=k$ and $M=1$. Also, note that $\ep_k \in \{-1,0\}$ implies that $-\frac12((\ep_k-1)^2+\ep_k-1)= \ep_k = -|\ep_k|$. The conclusion thus follows immediately from Theorem~\ref{thm:algorithm}.
\end{proof}

For powerfull-type fake~$\mu$'s, there is no general way to simplify the factorization of $D_f(s)$ obtained in Theorem~\ref{thm:algorithm}. For one thing, the initial index places no restriction at all on the critical index, as the following example shows.

\begin{ex}
Given integers $N>k\ge2$, suppose that
\[
\ep_j = \begin{cases}
1, &\text{if $j<N$ and $k\mid j$}, \\
0, &\text{if $j<N$ and $k\nmid j$}, \\
-1, &\text{if } j=N.
\end{cases}
\]
Then (regardless of the values of~$\ep_j$ for $j>N$) Algorithm~\ref{alg} terminates with $M=1$ and $c_1=k$ and $\ell=N$, so that the initial index is~$k$ and the critical index is~$N$.

In particular, for powerfull-type fake~$\mu$'s, it is possible for any index exceeding the initial index to be the critical index. (Proposition~\ref{prop:typeII} gives the same conclusion for powerfree-type fake~$\mu$'s. For M\"obius-type fake~$\mu$'s, the initial and critical indices always coincide by Proposition~\ref{prop:typeI}.)
\end{ex}

We can, however, give several examples of powerfull-type fake~$\mu$'s for which the critical index can be deduced.

\begin{ex}\label{ex:k+h}
Given integers $k>h\ge1$, suppose that
\[
\begin{cases}
\ep_j = 0, &\text{if } 1\le j\le k-1, \\
\ep_j = 1, &\text{if } j=k, \\
\ep_j \in \{0,1\}, &\text{if } k+1\le j\le k+h-1, \\
\ep_j = -1, &\text{if } j=k+h,
\end{cases}
\]
giving a powerfull-type fake~$\mu$ with initial index~$k$.
Algorithm~\ref{alg} reveals (regardless of the values of~$\ep_j$ for $j>k+h$) that the principal indices correspond to those~$k\le j\le k+h-1$ with $\ep_j=1$, and that the critical index equals~$k+h$; more precisely
\[
D_f(s) = \prod_{\substack{k\le j\le k+h-1 \\ \ep_j=1}} \zeta(js) \cdot\frac1{\zeta((k+h)s)} U_{k+h}(s)
\]
where the Dirichlet series coefficients of $U_{k+h}(s)$ are supported on $(k+h+1)$-free numbers.

Note that in this case, Theorems~\ref{thm:weak} and~\ref{thm:ubRH} imply that $E_f(x) = \Omega_{\pm}(x^{1/2(k+h)})$ and (under~RH) $E_f(x)\ll_\ep x^{1/2k+\ep}$. When~$h$ is small, these oscillation and upper bound results for $E_f(x)$ are close to each other. In particular, when $h\le\sqrt{2k}$, the oscillation result we obtain is better than the oscillation results mentioned in Example~\ref{ex:kfull} for the indicator function of the~$k$-full numbers themselves.
\end{ex}

\begin{ex}
Given integers $k'>k\ge2$ such that $k \nmid k'$, suppose that
\[
\ep_j = \begin{cases}
1, &\text{if $j = ak+bk'$ for some nonnegative integers $a$ and $b$}, \\
0, &\text{otherwise}.
\end{cases}
\]
Then the initial index is~$k$, and Algorithm~\ref{alg} yields the principal indices $\{k,k'\}$. Note that the only way to have $n_j>\ep_j$ in the second {\bf if} statement is for $n_j\ge2$, since by construction the number of representations of~$j$ from $\{k,k'\}$ is $n_j=0$ when $\ep_j=0$. Therefore the critical index of $(\ep_j)$ equals $\ell=\operatorname{lcm}[k,k']$, which is the smallest integer that can be written as a nonnegative integer combination of $\{k,k'\}$ in two different ways. In fact, since every integer of the form $ak+bk'$ can be written uniquely in this form with $0\le a<\frac\ell k$, in this case we have the exact zeta-factorization
\begin{align*}
D_f(s) &= \prod_p \biggl( 1+ \sum_{\substack{j\ge1 \\ j = ak+bk' \text{ for some } a,b\ge0}} \frac1{p^{js}} \biggr) \\
&= \prod_p \Biggl( \sum_{a=0}^{\ell/k-1} \sum_{b=0}^\infty \frac1{p^{(ak+bk')s}} \Biggr) = \prod_p \Biggl( \sum_{a=0}^{\ell/k-1} \frac1{p^{aks}} \Biggr)\prod_p \Biggl( \sum_{b=0}^\infty \frac1{p^{bk's}} \Biggr) = \frac{\zeta(ks)}{\zeta(\ell s)} \zeta(k's).
\end{align*}
One special case of this family is when $k=2$ and $k'=3$, so that $\ep_1=0$ while $\ep_j=1$ for all $j\ge2$, which recovers the factorization $\zeta(2s)\zeta(3s)/\zeta(6s)$ for the Dirichlet series corresponding to the indicator function of squarefull numbers.
\end{ex}

Our last example is an in-depth examination of the Dirichlet series corresponding to the indicator function of $k$-full numbers.
\begin{ex}\label{ex:kfullfactored}
Let $k \geq 3$, and let~$f$ be the indicator function of~$k$-full numbers, corresponding to the sequence with $\ep_j = 0$ when $j\le k-1$ and $\ep_j=1$ when $j\ge k$. Then~$D_f(s)$ has a zeta-factorization of the form
\begin{equation}\label{eq:kfullfactor}
D_f(s) = \sum_{n\text{ $k$-full}} n^{-s} = \prod_{j=k}^{2k-1} \zeta(js) \cdot \prod_{j=2k+2}^{4k+4} \zeta(js)^{a_j} \cdot U_{4k+4}(s),
\end{equation}
where the Dirichlet series coefficients of $U_{4k+4}$ are supported on $(4k+5)$-full numbers (a nearly equivalent statement appears as~\cite[Proposition 1(b)]{BRS88}, for instance).
We apply Theorem~\ref{thm:algorithm} to compute the exponents~$a_j$ explicitly. 

Following the notation used in Algorithm~\ref{alg}, observe that the principal indices are $(c_1,\ldots,c_k) = (k,k+1,\ldots,2k-1)$. Note that $2k=k+k$ and $2k+1=k+(k+1)$ have unique representations from $(c_1,\ldots,c_k)$, while $2k+2=k+(k+2) =2(k+1)$ has two representations from $(c_1,\ldots,c_k)$. Therefore the critical index equals $\ell=2k+2$, and $a_{2k}=\ep_{2k}-n_{2k}=1-1=0$ and similarly $a_{2k+1}=0$, while $a_{2k+2}=\ep_{2k+2}-n_{2k+2}=1-2=-1$. Note that the fraction in the second case of equation~\eqref{eq:a_j} equals~$0$, and so the formula for $a_{2\ell}=a_{4k+4}$ is the same sum as the formula for $a_{2k+3},\ldots,a_{4k+3}$.

For example, when $k=3$, the zeta-factorization from Theorem~\ref{thm:algorithm} turns out to be
$$
\sum_{n\text{ $3$-full}} n^{-s} = \frac{\zeta(3s)\zeta(4s)\zeta(5s)}{\zeta(8s)} \frac{\zeta(13s)\zeta(14s)}{\zeta(9s)\zeta(10s)} U_{16}(s),
$$
while when $k=4$, the zeta-factorization from Theorem~\ref{thm:algorithm} turns out to be
$$
\sum_{n\text{ $4$-full}} n^{-s} = \frac{\zeta(4 s)\zeta(5 s)\zeta(6 s)\zeta(7 s)}{\zeta(10 s)} \frac{\zeta(16 s)\zeta(17 s)^2\zeta(18 s)^2\zeta(19 s)^2\zeta(20s)}{\zeta(11 s)\zeta(12 s)^2\zeta(13 s)\zeta(14 s)} U_{20}(s).
$$
In both examples, the first fraction reflects the information in the previous paragraph, while the second fraction results from applying the formula in equation~\eqref{eq:a_j}.

Going forward, we assume $k \geq 5$. Recall that $\ep_j=1$ when $j \geq k$ or $j=0$, and $\ep_j=0$ otherwise; recall also that $c_i=k+i-1$ for $1\le i\le k$. To compute the exponents~$a_j$ in equation~\eqref{eq:kfullfactor} for $2k+3\leq j \leq 4k+4$, we apply equation~\eqref{eq:a_j}. which we can write in the form
\begin{equation}\label{eq:a_jformula}
a_j = \sum_{t=0}^k (-1)^t \bigl( S_t(j) + T_t(j) \bigr),    
\end{equation}
where
\begin{align*}
S_t(j) &= \#\bigg\{I \subset \{1,\ldots,k\},\, \#I=t\colon j-\sum_{i \in I} (k+i-1) \geq k\bigg\} \\
&= \#\bigg\{I \subset \{1,\ldots,k\},\, \#I=t\colon \sum_{i \in I} i \le j-t(k-1)-k \bigg\}, \\
T_t(j) &= \#\bigg\{I \subset \{1,\ldots,k\},\, \#I=t\colon j-\sum_{i \in I} (k+i-1) = 0 \bigg\} \\
&= \#\bigg\{I \subset \{1,\ldots,k\},\, \#I=t\colon \sum_{i \in I} i = j-t(k-1) \bigg\}.
\end{align*}
Note that we can restrict the sum in equation~\eqref{eq:a_jformula} to $0\le t\le 3$, since $j-t(k-1) \le (4k+4)-4(k-1) = 8$ when $t\ge4$ but the sum of any four distinct elements of $\{1,\ldots,k\}$ is at least~$10$.

Therefore computing $a_{2k+3},\ldots,a_{4k+4}$ reduces to finding, for $0\le t\le 3$, the number of $t$-element subsets of $\{1,\ldots,k\}$ whose sum equals or is bounded by particular numbers. This is an elementary but tedious problem, and we record only the results here, valid for all $k\ge5$ and all $2k+3\le j\le 4k+4$:
\begin{itemize}
\item $S_0(j)=1$ and $T_0(j)=0$;
\item $S_1(j) = \min\{k,j-2k+1\}$ and $T_1(j)=0$;
\item $\displaystyle S_2(j) = \begin{cases}
0, &\text{if } j\le 3k, \\
\lfloor\frac14(j-3k+1)^2\rfloor, &\text{if } 3k+1\le j\le4k-1, \\
\lfloor\frac14(j-3k+1)^2\rfloor - \binom{j-4k+2}2, &\text{if } j\ge4k;
\end{cases}$
\item $\displaystyle T_2(j) = \begin{cases}
\bigl\lfloor\frac12\bigl(k-|j-3k+1|\bigr)\bigr\rfloor, &\text{if } j\le4k-2, \\
0, &\text{if } j\ge4k-1;
\end{cases}$
\item $S_3(j) = \max\{0,j-4k-2\}$;
\item $\displaystyle T_3(j) = \begin{cases}
0, &\text{if } j\le3k, \\
[ \frac1{12}(j-3k)^2 ], &\text{if } 3k+1\le j\le4k, \\
[ \frac1{12}(j-3k)^2 ] - \lfloor\frac14(j-4k+1)^2\rfloor, &\text{if } j\ge4k+1.
\end{cases}$
\end{itemize}
For $T_3(j)$ we have used $[x]$ to denote rounding~$x$ to the nearest integer, in contrast to the greatest-integer function $\lfloor x\rfloor$ that also appears.

These formulas, together with equation~\eqref{eq:a_jformula}, allow for the full computation of the zeta-factor\-iz\-ation~\eqref{eq:kfullfactor} for the indicator function of the $k$-full numbers, for all $k\ge5$. For example, when $2k+3\le j\le 3k-1$ the formula~\eqref{eq:a_jformula} simplifies to $a_j = k-\lfloor\frac j2\rfloor$.
\end{ex}

\begin{rem} \label{interesting consequence}
The previous example has an interesting consequence for oscillations of the error term $E_f(x)$ of the counting function $F_f(x)$ for $k$-full numbers; let us restrict to $k\ge16$ for ease of exposition. In Example~\ref{ex:kfullfactored} we saw that the critical index of~$f$ equals $\ell=2k+2$, and thus Theorem~\ref{thm:weak} implies that $E_f(x) = \Omega_{\pm} (x^{1/(4k+4)})$. But in addition, Remark~\ref{tempting heuristic} described a tempting heuristic suggesting that these oscillations might be essentially best possible. We are now in a position to show that this heuristic does not hold here.

We computed in Example~\ref{ex:kfullfactored} that $a_{2k}=a_{2k+1}=0$ and $a_{2k+2}=-1$, and that $a_j=k-\lfloor \frac j2 \rfloor<0$ for $2k+3\leq j \leq 3k-1$. In particular, in view of equation~\eqref{eq:kfullfactor}, the Dirichlet series $D_f(s)$ is analytic at $\frac{1}{j}$ for all $2k \leq j \leq 3k-1$. Therefore the expression~\eqref{eq:G} for the main term $G_f(x)$ of the counting function $F_f(x)$ can be written as
\begin{align*}
G_f(x) &= \sum_{j=k}^{2k-1} \Res\biggl(D_f(s) \frac{x^s}{s}, \frac{1}{j} \biggr)+\sum_{j=3k}^{4k+4} \Res\biggl(D_f(s) \frac{x^s}{s}, \frac{1}{j} \biggr) = \sum_{j=k}^{2k-1} \Res\biggl(D_f(s) \frac{x^s}{s}, \frac{1}{j} \biggr) + O(x^{1/(3k-1)}).
\end{align*}
On the other hand, we saw in Example~\ref{ex:kfull} that
$$
F_f(x) = \sum_{j=k}^{2k-1} \Res\biggl(D_f(s) \frac{x^s}{s}, \frac{1}{j} \biggr)+\Omega\bigl(x^{1/(2k+\sqrt{8k}+3)} \bigr).
$$
Since $2k+\sqrt{8k}+3 < 3k-1$ when $k\ge16$, we conclude that $E_f(x) = F_f(x) - G_f(x) = \Omega(x^{1/(2k+\sqrt{8k}+3)})$, which is rather larger than the $x^{1/(4k+4)}$ suggested by the heuristic.

In summary, for a general $f \in \cF$, the heuristic prediction $E_f(x)\ll_\ep x^{1/2\ell+\ep}$ can be far away from the truth, and it is not immediately clear what the true order of magnitude of $E_f(x)$ should be in general.
\end{rem}

\section{Oscillation results for $E_f(x)$}\label{sec:omgea}

In this section, we prove a general oscillation result for the error term $E_f(x)$ of the summatory function $F_f(x)$ of a fake~$\mu$, based on the analytic properties of its Dirichlet series $D_f(s)$. More precisely, we state Theorem~\ref{thm:general} in Section~\ref{sec:strategy} and prove it in Section~\ref{sec:genosc}. From that theorem we then deduce Theorems~\ref{thm:mainmu} and~\ref{thm:mainkfree}, which apply to M\"obius-type and powerfree-type fake~$\mu$'s, respectively, as well as giving a general result (Theorem~\ref{thm:main_full}) for powerfull-type fake~$\mu$'s. Together these three results imply Theorem~\ref{thm:weak}.

\subsection{A motivating example and a general strategy} \label{sec:strategy}

As an illustration, we begin with the family of fake~$\mu$'s discussed in \cite[Theorem~3]{MMT23}, both to fill a gap in the original proof by Mossinghoff, Trudgian, and the first author, and to motivate our approach for the remainder of this section.

\begin{ex}\label{rem:gap}
Let $f \in \cF$ with $\ep_1=-1$ and $\ep_2=1$.
This family of M\"obius-type fake~$\mu$'s was studied in~\cite{MMT23}: it was shown there that $D_f(s) = {U_2(s)\zeta(2s)}/{\zeta(s)}$, where
$$
\quad U_2(s) = \prod_p \widetilde{C}_p(s) \quad\text{with}\quad
\widetilde{C}_p(s) =1+\sum_{j \geq 3} \frac{\ep_{j-1}+\ep_j}{p^{js}}.
$$
(Our Proposition~\ref{prop:typeI} contains this expression as a special case, other than the exact expression for~$\widetilde{C}_p(s)$ which, in this case, is easy to work out from the given zeta-factorization.)
Their result~\cite[Theorem~3]{MMT23}, translated into our notation, is that $E_f(x) = \Omega_{\pm}(\sqrt{x})$. We point out two small mistakes in their proof, partially to show how those gaps can be filled, and partially because describing the proof will help us motivate the approach we take in this section for general fake~$\mu$'s.

From an integral representation for $D_f(s) - c\zeta(2s)$ for real constants~$c$, they use a classical method of Landau to argue that if $E_f(x)$ is eventually of one sign then $D_f(s)$ must be analytic for $\sigma>\frac12$. Normally this would imply the Riemann hypothesis, as is stated in the first line of~\cite[page~3242]{MMT23}; but in this case, the fact that ${U_2(s)\zeta(2s)}/{\zeta(s)}$ is analytic for $\sigma>\frac12$ implies the slightly weaker statement that $\zeta(s)$ cannot have any zeros off the critical line {\em except possibly} at zeros of $U_2(s)$. While the authors of~\cite{MMT23} did not directly use the assertion of RH in their proof, this detail helps us see why it is important to compare the zeros of the various terms in a partial zeta-factorization (as we will do explicitly in Proposition~\ref{prop:nonzero} below).

The proof of~\cite[Theorem~3]{MMT23} is then carried out using the assertion that $U_2(\rho_1)\ne0$, where $\rho_1= \frac{1}{2}+i\gamma_1 \approx \frac12+i\cdot14.135$ is the lowest zero of $\zeta(s)$ in the upper half-plane. The authors claim that $U_2(\rho_1)\ne0$ follows from the fact that $U_2(s)$ is absolutely convergent for $\sigma>\frac13$, so that $1/U_2(s)$ is analytic there. However, this argument actually shows that $U_2(\rho_1)\ne0$ {\em unless} one of the individual Euler factors $\widetilde{C}_p(\rho_1)$ vanishes. (Indeed, the vanishing of an individual Euler factor is a significant contributor to~\cite[Theorem~1(iii)]{MMT23}.) It is impossible for any of these factors with $p\ge3$ to vanish, since by the triangle inequality
$$
\bigl| \widetilde{C}_p(\rho_1) \bigr| \geq 1-\sum_{j \geq 3} \frac{2}{p^{j/2}}=1-\frac{2}{p^{3/2}(1-p^{-1/2})} \ge 1-\frac{2}{3\sqrt3-3} > 0
$$
(this argument appears at the top of~\cite[page~3238]{MMT23}). We now show that $\widetilde{C}_2(\rho_1)\ne0$ to fill this small gap in their proof.

If we look at that Euler factor
$$
\widetilde{C}_2(\rho_1) =1 + \sum_{j\geq 3} \frac{\ep_{j-1} + \ep_j}{2^{j\rho_1}}
= 1 + \frac{1 + \ep_3}{2^{3\rho_1}} + \sum_{j=4}^6 \frac{\ep_{j-1} + \ep_j}{2^{j\rho_1}} + \sum_{j=7}^\infty \frac{\ep_{j-1} + \ep_j}{2^{j\rho_1}},
$$
then another triangle-inequality argument gives
$$
\bigl| \widetilde{C}_2(\rho_1) \bigl| \ge \biggl| 1 + \frac{1 + \ep_3}{2^{3\rho_1}} + \sum_{j=4}^6 \frac{\ep_{j-1} + \ep_j}{2^{j\rho_1}} \biggr| - \sum_{j=7}^\infty \frac2{2^{j/2}} = \biggl| 1 + \frac{1 + \ep_3}{2^{3\rho_1}} + \sum_{j=4}^6 \frac{\ep_{j-1} + \ep_j}{2^{j\rho_1}} \biggr| - \frac{1}{4\sqrt2-4}.
$$
We can split into~$81$ cases based on the values $\ep_3,\ldots,\ep_6\in\{-1,0,1\}$; in each case we directly compute the expression on the right-hand side, and in all~$81$ cases the quantity happens to exceed $0.055$. These inequalities prove that the Euler factor $\widetilde{C}_2(s)$ cannot vanish at $s=\rho_1$, which finishes the justification that $U_2(\rho_1) \neq 0$. 
\end{ex}

In Example~\ref{rem:gap}, the critical index was $\ell=1$, meaning that a power of $\zeta(s)$ was present in the denominator of our zeta-factorization; the fact that we were considering a zero of $\zeta(s)$ on the critical line is the reason why the geometric series we studied had common ratios whose size was~$p^{-1/2}$. For general critical indices~$\ell$, the corresponding series evaluated at a zero of $\zeta(\ell s)$ would have common ratios of size $p^{-1/2\ell}$ and would thus converge more slowly. In principle, we could try to extend the above triangle-inequality arguments, although at the very least we would have to consider more primes in addition to $p=2$ and more terms in each series; but such an extension would work only for one fixed~$\ell$ at a time in any case.

Instead, we exploit the fact that the specific zero~$\rho_1$ is not important to us---we can use any convenient zero of $\zeta(s)$, indeed without even needing to know its exact identity. We therefore adopt a suitably specified version of this strategy in Proposition~\ref{prop:nonzero} below. To facilitate precise statements of our results, we introduce some notation that will be in force throughout Section~\ref{sec:omgea}. We precede that notation with a lemma that will be familiar to analytic number theorists.

\begin{lem} \label{residue lemma}
Fix $\sigma_0>0$, and let $h(s)$ be a meromorphic function with a pole of order~$\xi$ at~$\sigma_0$. Then for any real number $x>1$, the residue of $h(s)x^s/s$ at $s=\sigma_0$ equals $P(\log x) x^{\sigma_0}$ for some polynomial~$P(t)$ (depending on~$h(s)$) whose degree is~$\xi-1$.
\end{lem}

\noindent
Note that the second sum in Theorem~\ref{thm:mainkfree} contains expressions of exactly this type coming from poles of order~$\xi=2$.

\begin{proof}
Since $h(s)/s$ is also meromorphic with a pole of order~$\xi$ at~$\sigma_0$, we can write
\begin{align*}
\frac{h(s)}s = \sum_{k=-\xi}^\infty c_k(s-\sigma_0)^k
\quad\text{and}\quad
x^s = x^{\sigma_0} e^{(s-\sigma_0)\log x} = x^{\sigma_0} \sum_{k=0}^\infty \frac{(s-\sigma_0)^k}{k!}(\log x)^k
\end{align*}
for some constants $c_k$ with $c_{-\xi}\ne0$. When we multiply these series together, the coefficient of $(s-\sigma_0)^{-1}$ depends on the first~$\xi$ terms in each series; more precisely, the residue we seek equals
\[
x^{\sigma_0} \sum_{j=0}^{\xi-1} \frac{(\log x)^j}{j!} c_{-1-j} = P(\log x)x^{\sigma_0}
\quad\text{where}\quad
P(t) = \sum_{j=0}^{\xi-1} \frac{c_{-1-j}}{j!} t^j.
\qedhere
\]
\end{proof}

\begin{notn} \label{sec 3 notation}
Let $f\in \cF$ be defined via the sequence $(\ep_j)$. Let $\ell$ be the critical index of~$f$. By Proposition~\ref{prop:typeI} and Theorem~\ref{thm:algorithm}, there are integers $a_1, a_2, \ldots, a_{2\ell}$ and a function $U_{2\ell}(s)$ such that the following conditions hold:
\begin{itemize}
\item We have $a_j\geq 0$ for $1\leq j\le\ell-1$, while $a_\ell<0$;
\item For $\sigma>1$, we can write
\begin{equation}\label{eq:factorization}
D_f(s) =U_{2\ell}(s) \cdot \prod_{j=1}^{2\ell} \zeta(js)^{a_j}
\quad{where }\quad
U_{2\ell}(s) = \prod_{p} \biggl(1+\sum_{j=2\ell+1}^{\infty} \frac{\eta_j}{p^{js}} \biggr);
\end{equation}
\item There are positive constants $A,B$, depending only on~$\ell$, such that $|\eta_j|\leq (Aj)^B$ for all~$j$; in particular, $U_{2\ell}(s)$ is analytic for $\sigma>\frac1{2\ell+1}$ by Lemma~\ref{really does converge lemma}, and is nonzero provided each of its factors is nonzero.
\item Given equation~\eqref{eq:factorization}, in the region $\sigma>\frac{1}{2\ell+1}$, the only possible real poles of $D_f(s)$ are at $s=1, \frac{1}{2}, \ldots, \frac{1}{2\ell}$. Let $\xi_j$ be the order of the pole of $D_f(s)$ at $s=\frac1j$, so that $\xi_j=0$ if $a_j\le0$ and $0\le \xi_j\le a_j$ if $a_j\ge1$.
\item By Lemma~\ref{residue lemma}, equation~\eqref{eq:G} becomes
\[
G_{f}(x) = \sum_{j=1}^{2\ell} \Res\biggl(D_f(s) \frac{x^s}{s}, \frac{1}{j} \biggr) = \sum_{\substack{ 1\leq j \leq 2\ell \\ \xi_j \geq 1}} P_j(\log x) x^{1/j},
\]
where each $P_j(t)$ is a polynomial of degree $\xi_j-1$.
\item As always, we have $F_f(x) = \sum_{n\le x} f(n)$ and $E_f(x) = F_f(x) - G_f(x)$.
\end{itemize}
\end{notn}

We are now in a position to state an important proposition, which asserts that a zero of $\zeta(s)$ exists at which other zeta-factorization factors do not vanish, as motivated by the above discussion. Establishing this proposition is the goal of Section~\ref{nonvanish section}.

\begin{prop}\label{prop:nonzero}
In the situation described by Notation~\ref{sec 3 notation}, there exists a zero~$\rho$ of $\zeta(s)$ with $\Re(\rho) \geq \frac{1}{2}$ such that
\begin{equation}\label{nonzero}
U_{2\ell}\biggl( \frac{\rho}{\ell} \biggr) \cdot \prod_{\substack{1\leq j \leq 2\ell\\j \neq \ell}}\zeta\biggl( \frac{j\rho}{\ell} \biggr)\neq 0.
\end{equation}
\end{prop}

That proposition will allow us to prove the following general oscillation theorem in Section~\ref{sec:genosc}:

\begin{thm}\label{thm:general}
In the situation described by Notation~\ref{sec 3 notation},
$$E_f(x) = \Omega_{\pm} \bigl(x^{1/2\ell}(\log x)^{|a_\ell|-1} \bigr).$$
\end{thm}

\subsection{Finding a zero where other factors do not vanish} \label{nonvanish section}

This section is devoted to the proof of Proposition~\ref{prop:nonzero}.
We begin by recording some consequences of classical zero-counting functions for $\zeta(s)$ and the Landau--Gonek formula.

\begin{lem}\label{lem:Zl}
Given $\ell\in\N$ and $T\ge3$, define
\begin{multline*}
Z_\ell(T) = \bigg\{\rho= \beta+i\gamma\colon \zeta(\rho) =0,\, 0< \gamma \leq T,\, \frac{\ell}{2\ell+1}<\beta<\frac{\ell+1}{2\ell+1}, \\
\prod_{\substack{1\leq j \leq 2\ell\\j \neq \ell}}\zeta\biggl( \frac{j\rho}{\ell} \biggr) \zeta\biggl( \frac{j(1-\rho)}{\ell} \biggr)\neq 0\bigg\}.
\end{multline*}
Then $\#Z_\ell(T) =\frac{T}{2\pi} \log T+O_\ell(T)$.
\end{lem}

\begin{proof}
In the usual notation
\begin{align*}
N (T) &= \# \{\rho= \beta+i\gamma\colon \zeta(\rho) =0,\, 0< \gamma \leq T\} \\
N (\sigma, T) &= \# \{\rho= \beta+i\gamma\colon \zeta(\rho) =0,\, 0< \gamma \leq T, \, \beta\ge\sigma\},
\end{align*}
we know that $N(T) = \frac{T}{2\pi} \log T + O(T)$ (see for example~\cite[Corollary 14.2]{MV07}), while
well-known zero-density estimates (first proved by Bohr and Landau~\cite{BL14}) imply that $N(\sigma, T)\ll_\sigma T$ for $\sigma>\frac{1}{2}$. We use this latter estimate to bound how many of the $N(T)$ zeros up to height~$T$ are not included in $Z_\ell(T)$. For the rest of this proof, all implicit constants may depend on~$\ell$.

The upper bound $\beta<\frac{\ell+1}{2\ell+1}$ excludes $N(\frac{\ell+1}{2\ell+1},T) \ll T$ zeros; also, by the symmetries of $\zeta(s)$, the lower bound $\beta>\frac{\ell}{2\ell+1}$ excludes $N(1-\frac{\ell}{2\ell+1},T) \ll T$ zeros. Suppose that~$\rho$ is a zero that has not been excluded so far. If $\ell+1\le j\le2\ell$, then $\Re(\frac{j}{\ell}\rho) \ge \frac{\ell+1}\ell \frac\ell{2\ell+1} = \frac{\ell+1}{2\ell+1} > \frac12$, and so the number of these $\frac{j}{\ell}\rho$ that can be zeros of $\zeta(s)$ is at most $N(\frac{\ell+1}{2\ell+1},T) \ll T$. Similarly, if $1\le j\le\ell-1$, then $\Re(\frac{j}{\ell}\rho) \le \frac{\ell-1}\ell \frac{\ell+1}{2\ell+1} < \frac12$, and so the number of these $\frac{j}{\ell}\rho$ that can be zeros of $\zeta(s)$ is at most $N(1-\frac{\ell-1}\ell \frac{\ell+1}{2\ell+1},T) \ll T$.
Similar observations holds for the $\zeta\bigl( \frac{j(1-\rho)}{\ell} \bigr)$ factors.
We conclude that $\#Z_\ell(T) =N(T)+O_\ell(T) =\frac{T}{2\pi} \log T+O_\ell(T)$ as desired.
\end{proof}

\begin{lem}\label{lem:Landau}
If $x>0$ with $x \neq 1$, then 
$$
\sum_{0< \gamma \leq T} x^\rho \ll_x T.
$$    
\end{lem}

\begin{proof}
Landau's formula~\cite{L12} tells us that if $x>1$,
$$
\sum_{0< \gamma \leq T} x^\rho=-\frac{T}{2\pi}\Lambda(x)+O(\log T),
$$
where $\Lambda(x)$ is the von Mangoldt function when~$x$ is an integer and $\Lambda(x)=0$ otherwise.
When $0<x<1$, one can derive from Landau's formula (see for example~\cite[``Corollary'']{G93}) that
$$
\sum_{0< \gamma \leq T} x^\rho=-\frac{Tx}{2\pi}\Lambda\biggl( \frac{1}{x} \biggr)+O(\log T).
$$
In both cases the right-hand side is $\ll_x T$ as required.
\end{proof}

It will be helpful to give names to the Euler factors appearing in Notation~\ref{sec 3 notation}.

\begin{notn} \label{Cp notation}
For each prime~$p$, define
$$
C_p(s) =1+\sum_{j=1}^{\infty} \frac{\ep_j}{p^{js}}
\quad\text{and}\quad
\widetilde{C}_p(s) =1+\sum_{j=2\ell+1}^{\infty} \frac{\eta_{j}}{p^{js}},
$$
so that $D_f(s) = \prod_{p} C_p(s)$ for $\sigma>1$ and $U_{2\ell}(s) = \prod_{p} \widetilde{C}_p(s)$ for $\sigma>\frac{1}{2\ell+1}$. Note that the bounds $|\ep_j| \le 1$ and $|\eta_j| \le (Aj)^B$ imply that the series defining $C_p(s)$ and $\widetilde C_p(s)$ both converge for $\sigma>0$.
\end{notn}

\begin{lem}\label{lem:deduction}
In the situation described by Notation~\ref{sec 3 notation}, there exists a positive integer~$P$, depending only on~$\ell$, such that for each zero $\rho$ of $\zeta$ with $\Re(\rho)>\frac{\ell}{2\ell+1}$,
\[
U_{2\ell}\biggl( \frac{\rho}{\ell} \biggr)\neq 0
\quad\text{if and only if}\quad
\prod_{p \leq P} C_p\biggl( \frac{\rho}{\ell} \biggr)\neq 0.
\]
\end{lem}

\begin{proof}
We certainly have
\[
\biggl| 1 - \widetilde{C}_p\biggl( \frac{\rho}{\ell} \biggr) \biggr| \le \sum_{j=2\ell+1}^{\infty} \frac{(Aj)^B}{p^{j\Re\rho/\ell}} < \sum_{j=2\ell+1}^{\infty} \frac{(Aj)^B}{p^{j/(2\ell+1)}}.
\]
The right-hand side is a positive series that is decreasing in~$p$ and tends to~$0$ termwise; by the monotone convergence theorem, the series itself tends to~$0$ as $p \to \infty$. In particular, there exists a positive integer~$P$ such that the series is less than~$1$ when $p\ge P$, and for these primes we deduce that $\widetilde{C}_p(\rho/\ell) \ne 0$. Moreover, since~$A$ and~$B$ depend only on~$\ell$, the same is true of~$P$.

As we observed in Notation~\ref{sec 3 notation}, $U_{2\ell}(\rho/\ell) =0$ only if $\widetilde{C}_p(\rho/\ell) =0$ for some prime~$p$, and such a prime must necessarily be less than~$P$. On the other hand, observe that equation~\eqref{eq:factorization} implies that
$\widetilde{C}_p(s) = C_p(s) \prod_{j=1}^{2\ell} (1-p^{-js})^{-a_j}$ for $\Re s>1$, and so this identity remains true in the larger-half plane $\Re s>0$ where all terms converge. In particular,
$$
\widetilde{C}_p\biggl( \frac{\rho}{\ell} \biggr) = C_p\biggl( \frac{\rho}{\ell} \biggr) \prod_{j=1}^{2\ell} (1-p^{-j\rho/\ell})^{-a_j},
$$
and thus $\widetilde{C}_p(\rho/\ell) =0$ if and only if $C_p(\rho/\ell) =0$, which completes the proof of the lemma.
\end{proof}

We now have all the ingredients to establish Proposition~\ref{prop:nonzero}.

\begin{proof}[Proof of Proposition~\ref{prop:nonzero}]
For notational convenience, set $\ep_0=1$. Note that uniformly for all primes~$p$, all positive integers~$n$, and all zeros~$\rho$ of~$\zeta(s)$ with $\Re(\rho)\ge\frac{\ell}{2\ell+1}$,
\begin{align*}
C_p\biggl( \frac{\rho}{\ell} \biggr) &= \sum_{j=0}^n \frac{\ep_j}{p^{{j\rho}/{\ell}}} + O\biggl( \sum_{j=n+1}^\infty \frac1{p^{{j\Re\rho}/{\ell}}} \biggr) \\
&= \sum_{j=0}^n \frac{\ep_j}{p^{{j\rho}/{\ell}}} + O\biggl( \sum_{j=n+1}^\infty \frac1{p^{{j}/{(2\ell+1})}} \biggr) = \sum_{j=0}^n \frac{\ep_j}{p^{{j\rho}/{\ell}}} + O\biggl( \frac{1}{p^{(n+1)/(2\ell+1)}} \biggr).
\end{align*}

Let $P$ be the positive integer (depending only on~$\ell$) from Lemma~\ref{lem:deduction}, and set
$$
Q_n = \biggl\{ \prod_{p \leq P} p^{\alpha(p)}\colon 0 \leq \alpha(p)\leq n \biggr\}.
$$
Note that $|Q_n|\leq (n+1)^P$.
By expanding the product, we see that uniformly for all $n\in\N$ and for all zeros~$\rho$ of $\zeta(s)$ with $\frac{\ell}{2\ell+1}<\Re(\rho)<\frac{\ell+1}{2\ell+1}$,
\begin{align*}
\prod_{p \leq P} & C_p\biggl( \frac{\rho}{\ell} \biggr) C_p\biggl( \frac{1-\rho}{\ell} \biggr) \\
&= \prod_{p \leq P} \biggl( \sum_{j=0}^n \frac{\ep_j}{p^{{j\rho}/{\ell}}}+O\biggl( \frac{1}{p^{(n+1)/(2\ell+1)}} \biggr)\biggr) \biggl( \sum_{j=0}^n \frac{\ep_j}{p^{{j(1-\rho)}/{\ell}}}+O\biggl( \frac{1}{p^{(n+1)/(2\ell+1)}} \biggr)\biggr) \\
&= \biggl( \sum_{q \in Q_{n}} \frac{f(q)}{q^{\rho/\ell}} \biggr) \biggl( \sum_{q \in Q_{n}} \frac{f(q)}{q^{(1-\rho)/\ell}} \biggr)+O\biggl( \frac{(n+1)^{2P}-1}{2^{(n+1)/(2\ell+1)}} \biggr).
\end{align*}
In particular, we can choose~$n$ sufficiently large in terms of~$\ell$ so that
\begin{equation} \label{eq:expansion}
\biggl| \prod_{p \leq P} C_p\biggl( \frac{\rho}{\ell} \biggr) C_p\biggl( \frac{1-\rho}{\ell} \biggr) \biggr| \ge \biggl| \biggl( \sum_{q \in Q_{n}} \frac{f(q)}{q^{\rho/\ell}} \biggr) \biggl( \sum_{q \in Q_{n}} \frac{f(q)}{q^{(1-\rho)/\ell}} \biggr) \biggr| - \frac12.
\end{equation}

For each nontrivial zero $\rho$ of $\zeta$, we have
\begin{align*}
\biggl( \sum_{q \in Q_{n}} \frac{f(q)}{q^{\rho/\ell}} \biggr) \biggl( \sum_{q \in Q_{n}} \frac{f(q)}{q^{(1-\rho)/\ell}} \biggr) &= \sum_{q_1,q_2 \in Q_{n}} \frac{f(q_1)f(q_2)}{q_2^{1/\ell}} \biggl( \frac{q_2}{q_1} \biggr)^{\rho/\ell} \\
&= \sum_{q \in Q_{n}} \frac{f(q)^2}{q^{1/\ell}}+ \sum_{\substack{q_1,q_2 \in Q_{n}\\ q_1 \neq q_2}} \frac{f(q_1)f(q_2)}{q_2^{1/\ell}} \biggl( \frac{q_2}{q_1} \biggr)^{\rho/\ell}.
\end{align*}
By Lemma~\ref{lem:Landau}, for each $q_1,q_2 \in Q_n$ with $q_1 \neq q_2$,
$$
\sum_{0<\gamma \leq T} \biggl( \frac{q_2}{q_1} \biggr)^{\rho/\ell} \ll T;
$$
while the implicit constant depends on~$q_1$ and~$q_2$, both numbers come from the finite set~$Q_n$ which depends only on~$\ell$. We conclude that
\begin{align*} 
\sum_{0<\gamma \leq T} \biggl( \sum_{q \in Q_{n}} \frac{f(q)}{q^{\rho/\ell}} \biggr) \biggl( \sum_{q \in Q_{n}} \frac{f(q)}{q^{(1-\rho)/\ell}} \biggr) &= \biggl( \sum_{q \in Q_{n}} \frac{f(q)^2}{q^{1/\ell}} \biggr) N(T) + O_\ell(T) \\
&\ge 1\cdot N(T) + O_\ell(T) = \frac{T}{2\pi} \log T + O_\ell(T),
%\label{eq:sumQn}
\end{align*}
since $f(1)=1$.
Also, note that for each nontrivial zero~$\rho$ of $\zeta(s)$, the trivial bound $|f(q)|\le1$ implies that
$$
\biggl|\biggl( \sum_{q \in Q_{n}} \frac{f(q)}{q^{\rho/\ell}} \biggr) \biggl( \sum_{q \in Q_{n}} \frac{f(q)}{q^{(1-\rho)/\ell}} \biggr)\biggr|\leq |Q_n|^2 \leq (n+1)^{2P}.
$$
Since the set $Z_\ell(T)$ defined in Lemma~\ref{lem:Zl} excludes $\ll_\ell T$ zeros up to height~$T$, it follows that
\begin{align*}
\sum_{\rho \in Z_\ell(T)} \biggl( \sum_{q \in Q_{n}} \frac{f(q)}{q^{\rho/\ell}} \biggr) \biggl( \sum_{q \in Q_{n}} \frac{f(q)}{q^{(1-\rho)/\ell}} \biggr) &= \sum_{0<\gamma \leq T} \biggl( \sum_{q \in Q_{n}} \frac{f(q)}{q^{\rho/\ell}} \biggr) \biggl( \sum_{q \in Q_{n}} \frac{f(q)}{q^{(1-\rho)/\ell}} \biggr) + O_\ell(T (n+1)^{2P}) \\
&\ge \frac{T}{2\pi} \log T + O_\ell(T).
\end{align*}
Finally, equation~\eqref{eq:expansion} implies that
$$
\sum_{\rho \in Z_\ell(T)} \biggl( \prod_{p \leq P} C_p\biggl( \frac{\rho}{\ell} \biggr) C_p\biggl( \frac{1-\rho}{\ell} \biggr)\biggr) \ge \frac{T}{2\pi} \log T  + O_\ell(T) - \frac12 \#Z_\ell(T) \ge  \frac{T}{4\pi} \log T + O_\ell(T).
$$
In particular, when~$T$ is sufficiently large in terms of~$\ell$, there exists $\rho\in Z_\ell(T)$ such that  
$$
\prod_{p \leq P} C_p\biggl( \frac{\rho}{\ell} \biggr) C_p\biggl( \frac{1-\rho}{\ell} \biggr)\neq 0,
$$
which implies that
$$
U_{2\ell}\biggl( \frac{\rho}{\ell} \biggr)U_{2\ell}\biggl( \frac{1-\rho}{\ell} \biggr) \neq 0
$$
by Lemma~\ref{lem:deduction}.

Since $\rho\in Z_\ell(T)$ implies $1-\rho\in Z_\ell(T)$, we may assume that $\Re\rho \ge \frac12$. Then $\zeta(\rho)=0$ and 
\[
\prod_{\substack{1\leq j \leq 2\ell\\j \neq \ell}}\zeta\biggl( \frac{j\rho}{\ell} \biggr)\neq 0.
\]
by the definition of $Z_\ell(T)$, and we have just shown that $U_{2\ell}\bigl( \frac{\rho}{\ell} \bigr) \ne 0$,
which confirms equation~\eqref{nonzero} and hence establishes the proposition.
\end{proof}

\subsection{The general oscillation result} \label{sec:genosc}

We are now able to establish our most general oscillation result: in the situation described by Notation~\ref{sec 3 notation} (which is used throughout the proof), we need to show that
$E_f(x) = \Omega_{\pm} \bigl(x^{1/2\ell}(\log x)^{|a_\ell|-1} \bigr)$.

\begin{proof}[Proof of Theorem~\ref{thm:general}]
We show that $E_f(x) = \Omega_- \bigl(x^{1/2\ell}(\log x)^{|a_\ell|-1} \bigr)$, as the proof of the corresponding $\Omega_+$ result is almost identical.
For each $1\leq j \leq 2\ell$ with $\xi_j\ge1$, let the coefficients $b_{j,k}$ for $0\le k\le\xi_j-1$ be defined by $P_j(y) = \sum_{k=0}^{\xi_j-1} b_{j,k} y^{k}$. Define
$$
\widetilde{D}(s) =D_f(s)-\sum_{\substack{ 1\leq j \leq 2\ell \\ \xi_j \geq 1}} \sum_{k=0}^{\xi_j-1} \frac{ b_{j,k}s \cdot k!}{(s-\frac{1}{j})^{k+1}}.
$$
Note that for each $k \geq 0$,
$$
\frac{k!}{(s-\frac{1}{j})^{k+1}}= \int_{1}^{\infty} \frac{x^{1/j}(\log x)^k}{x^{s+1}} dx.
$$
It follows that
$$
\widetilde{D}(s) =s \int_{1}^{\infty} \frac{E_f(x)}{x^{s+1}},
$$
and thus $\widetilde{D}(s)$ has no real pole with $s\geq \frac{1}{2\ell}$. 

Let $r>0$ be a constant to be chosen later, and define
$$
H(s) = \widetilde{D}(s)+\frac{rs (|a_\ell|-1)!}{(s-\frac{1}{2\ell})^{-a_\ell}}.
$$
Then for $\sigma>1$,
\begin{equation}\label{eq:difference}
H(s) =s \int_{1}^{\infty} \frac{E_f(x)+rx^{1/2\ell}(\log x)^{|a_\ell|-1}}{x^{s+1}}\,dx. 
\end{equation}

Suppose that $E_f(x)+rx^{1/2\ell}(\log x)^{|a_\ell|-1}$ is positive when~$x$ is sufficiently large. Note that $\widetilde{D}(s)$ has no real singularity with $s\geq \frac{1}{2\ell}$, and the smallest real singularity of $(s-\frac{1}{2\ell})^{a_\ell}$ is at $s= \frac{1}{2\ell}$. Thus, Landau's theorem (see for example~\cite[Lemma 15.1]{MV07}) implies that $H(s)$ is analytic for $\sigma>\frac{1}{2\ell}$ and equation~\eqref{eq:difference} holds for $\sigma>\frac{1}{2\ell}$. On the other hand, by Proposition~\ref{prop:nonzero}, there is a nontrivial zero $\rho$ of $\zeta$ such that $\Re(\rho)\geq \frac{1}{2}$ and 
\begin{equation}\label{eq:nonzero}
U_{2\ell}\biggl( \frac{\rho}{\ell} \biggr) \cdot \prod_{\substack{1\leq j \leq 2\ell\\j \neq \ell}}\zeta\biggl( \frac{j\rho}{\ell} \biggr)\neq 0.    
\end{equation}
In particular, since $\zeta(\ell s)$ has a zero at $\rho/\ell$, $D_f(s)$ does indeed have a pole at $\rho/\ell$ by equation~\eqref{eq:factorization}, and thus so does $H(s)$. Since $H(s)$ is analytic for $\sigma>\frac{1}{2\ell}$, we deduce that $\Re(\rho) = \frac{1}{2}$. Set $\rho= \frac{1}{2}+i\gamma$ and $\gamma'= \gamma/\ell$. Equation~\eqref{eq:difference} implies that for $\sigma>\frac{1}{2\ell}$,
\begin{align*}
|H(\sigma+i \gamma')|
\leq |\sigma+i \gamma'| \int_{1}^{\infty} \frac{\big|E_f(x)+rx^{1/2\ell}(\log x)^{|a_\ell|-1} \big|}{x^{\sigma+1}}\,dx= \frac{|\sigma+i \gamma'|}{\sigma} H(\sigma) < 2\ell|\rho| H(\sigma).    
\end{align*}

Let $m$ be the order of~$\rho$ as a zero of $\zeta(s)$, and set $m'=|a_\ell| m$. The above inequality implies that
\begin{equation}\label{eq:ineq}
\lim_{\sigma \to \frac{1}{2\ell}^+} \biggl( \sigma-\frac{1}{2\ell} \biggr)^{m'}
|H(\sigma+i \gamma')|
\leq 2\ell|\rho|\lim_{\sigma \to \frac{1}{2\ell}^+} \biggl( \sigma-\frac{1}{2\ell} \biggr)^{m'} H(\sigma).
\end{equation}
Since $\widetilde{D}(s)$ is analytic at $\sigma= \frac{1}{2\ell}$, the right-hand side of inequality~\eqref{eq:ineq} is equal to
$$
2\ell|\rho|\lim_{\sigma \to \frac{1}{2\ell}^+} \biggl( \sigma-\frac{1}{2\ell} \biggr)^{m'} \biggl|\frac{r\sigma(|a_\ell|-1)!}{(\sigma-\frac{1}{2\ell})^{-a_\ell}} \biggr|=
\begin{cases}
    r|\rho|(|a_\ell|-1)!, &\text{if } m=1,\\
    0, &\text{if } m>1.
\end{cases}
$$
Since $\rho$ is a zero of $\zeta(s)$ of order $m$, the left-hand side of inequality~\eqref{eq:ineq} is equal to
\begin{align*}
\lim_{\sigma \to \frac{1}{2\ell}^+} & \biggl( \sigma-\frac{1}{2\ell} \biggr)^{m'}
\biggl|D_f(\sigma+i \gamma')-\sum_{\substack{ 1\leq j \leq 2\ell \\ \xi_j \geq 1}} \sum_{k=0}^{\xi_j-1} \frac{ b_{j,k}s \cdot k!}{(\sigma+i \gamma'-\frac{1}{j})^{k+1}}+\frac{r(\sigma+i \gamma')(|a_\ell|-1)!}{(\sigma+i \gamma'-\frac{1}{2\ell})^{-a_\ell}} \biggr|\\
&= \lim_{\sigma \to \frac{1}{2\ell}^+} \biggl( \sigma-\frac{1}{2\ell} \biggr)^{m'} \big|D_f(\sigma+i \gamma')\big|= \biggl( \frac{m!}{\ell^m |\zeta^{(m)}(\rho)|} \biggr)^{-a_\ell} \biggl|U_{2\ell}\biggl( \frac{\rho}{\ell} \biggr)\biggr| \prod_{\substack{1\leq j \leq 2\ell \\ j \neq \ell}} \biggl|\zeta\biggl( \frac{j\rho}{\ell} \biggr)\biggr|^{a_j}.    
\end{align*}
Now inequality~\eqref{eq:nonzero} implies that $m=1$ and thus 
\begin{equation}\label{eq:lb}
\frac{|U_{2\ell}(\frac{\rho}{\ell})|}{(\ell |\zeta'(\rho)|)^{-a_\ell}} \cdot \prod_{\substack{1\leq j \leq 2\ell \\ j \neq \ell}} \biggl|\zeta\biggl( \frac{j\rho}{\ell} \biggr)\biggr|^{a_j} \leq r|\rho|(|a_\ell|-1)!.    
\end{equation}
Note that inequality~\eqref{eq:nonzero} implies that the left-right side of inequality~\eqref{eq:lb} is nonzero, and thus inequality~\eqref{eq:lb} can only hold when~$r$ is sufficiently large. In other words, for smaller positive values of~$r$, the difference $E_f(x)+rx^{1/2\ell}(\log x)^{|a_\ell|-1}$ cannot be always positive for sufficiently large~$x$, which shows that $E_f(x) = \Omega_- \bigl(x^{1/2\ell}(\log x)^{|a_\ell|-1} \bigr)$ as required.
\end{proof}

\subsection{Applications} \label{sec:applications}

It is now a simple matter to use Theorem~\ref{thm:general} to derive Theorems~\ref{thm:mainmu} and~\ref{thm:mainkfree}, which apply to M\"obius-type and powerfree-type fake~$\mu$'s, respectively. We also use Theorem~\ref{thm:general} to give a general result (Theorem~\ref{thm:main_full} below) that is our strongest oscillation result for powerfull-type fake~$\mu$'s. Together these three results imply Theorem~\ref{thm:weak}.

\begin{proof}[Proof of Theorem~\ref{thm:mainmu}]
For each $1\leq j \leq 2k$, Proposition~\ref{prop:typeI} tells us that $D_f(s)$ has at most a simple pole at $s= \frac{1}{j}$; therefore the residue of $D_f(s) \cdot \frac{x^s}{s}$ at $s= \frac{1}{j}$ is $\faf(j)x^{1/j}$ where $\faf(j)$ is the constant from Definition~\ref{defn:doublepole}. The theorem then follows from Theorem~\ref{thm:general} by applying the partial zeta-factorization of $D_f(s)$ in Proposition~\ref{prop:typeI}.    
\end{proof}

Combining Theorem~\ref{thm:algorithm} and Theorem~\ref{thm:general} results immediately in the following theorem.

\begin{thm}\label{thm:main_full}
Let $f \in \cF$ be of powerfree-type or powerfull-type. In the notation in Theorem~\ref{thm:algorithm},
\begin{align*}
G_f(x) &= \sum_{j=1}^{M} \Res\biggl(D_f(s) \cdot \frac{x^s}{s}, \frac{1}{c_j} \biggr)+
\sum_{j= \ell+1}^{2\ell} \Res\biggl(D_f(s) \cdot \frac{x^s}{s}, \frac{1}{j} \biggr) \\
E_f(x) &= \Omega_{\pm} (x^{{1}/{2\ell}} (\log x)^{n_{\ell}-\ep_{\ell}-1}).
\end{align*}
\end{thm}

\begin{proof}[Proof of Theorem~\ref{thm:mainkfree}]
Let $1 \leq j \leq 2k$. By Proposition~\ref{prop:typeII}, $D_f(s)$ has a pole at $s= \frac{1}{j}$ with order at most $2$. By Definition~\ref{defn:doublepole}, the principal part of $D_f(s)$ is
\[
\frac{\fbf(j)}{j^2(s-1/j)^2} + \frac{\faf(j)}{j(s-1/j)}.
\]
Thus the residue of $D_f(s) \cdot \frac{x^s}{s}$ at $s= \frac{1}{j}$ is given by
$$
\lim_{s \to \frac{1}{j}} \frac{d}{ds} D_f(s) \cdot \frac{x^s}{s}= \lim_{s \to \frac{1}{j}} \frac{d}{ds} \biggl( \frac{\fbf(j)}{j^2(s-1/j)^2} + \frac{\faf(j)}{j(s-1/j)} \biggr) \cdot \frac{x^s}{s}= \faf(j)x^{1/j}+\fbf(j) x^{1/j} \biggl( \frac{\log x}{j}-1\biggr).
$$
The theorem now follows from Proposition~\ref{prop:typeII} and Theorem~\ref{thm:main_full}. 
\end{proof}

\section{Upper bounds on $E_f(x)$}\label{sec:ub}

In this section, we prove our upper bounds on $E_f(x)$, both unconditional (Theorem~\ref{thm:unconditional}, which is proved in Section~\ref{sec:conv}) and assuming RH (Theorem~\ref{thm:ubRH}, which is proved for powerfree-type~$f$ in Section~\ref{sec:conv} and for M\"obius- and powerfull-type~$f$ in Section~\ref{sec:conditional2}). Our main motivation for establishing these upper bounds is to provide some sort of calibration against which to gauge the strength of our oscillation results. As it happens, this exercise also allows us to gather techniques from the literature and generalize their scope to all fake~$\mu$'s; in doing so, we have often recovered, and sometimes even improved, the best known upper bounds for error terms in special cases. We will combine two different methods to prove upper bounds on $E_f(x)$, namely the convolution method (or Dirichlet hyperbola method) and the method of contour integration.

\subsection{Upper bounds on $E_f(x)$ for a special family of fake~$\mu$'s}\label{sec:gk}

In this section, we prove Theorem~\ref{thm:unconditional} and Theorem~\ref{thm:ubRH} for a special family of fake~$\mu$'s of powerfree-type. 

\begin{defn} \label{k and g def}
For each $k\ge1$, let $h_k(n)$ be the multiplicative function appearing in the Dirichlet series $\zeta(ks)^{-2} = \sum_{n=1}^\infty h_k(n) n^{-s}$. We can check that $h_k(p^k)=-2$ (so that~$h_k$ is not quite a fake~$\mu$) and $h_k(p^{2k})=1$ and that $h_k(p^j)=0$ for all other $j\ge1$. Define $H_k(x)=\sum_{n \leq x}h_k(n)$.

Moreover, for each $k\ge1$, let $g_k$ be the fake~$\mu$ defined via the sequence $(\ep_j)_{j=1}^{\infty}$ with $\ep_j=1$ for $1\leq j \leq k-1$ and $\ep_j=-1$ for $k\leq j \leq 2k-1$ and $\ep_j=0$ for $j\ge 2k$. We can check that $g_k$ is the Dirichlet convolution of~$h_k$ and the constant function~$1$, which is the same as saying that the sequence $(1,\ep_1,\ep_2,\ldots)$ is the convolution of $(1,1,1,\ldots)$ and $(1,h_k(p^1),h_k(p^2),\ldots)$;
in particular, $D_{g_k}(s) = \sum_{n=1}^\infty g_k(n) n^{-s} = \zeta(s)\zeta(ks)^{-2}$. Note that $g_1=\mu$.
\end{defn}

The proof of our more general results for all fake~$\mu$'s of powerfree-type will build on the upper bounds on $E_{g_k}(x)$ in Propositions~\ref{prop:conv-1} and~\ref{prop:conv-1RH} below. As preparation, we need the following lemmas. In this section, $\tau(n)$ denotes the number of positive divisors of~$n$.

\begin{lem}\label{lem:h1prepre}
We have $|h_1(n)| \le \tau(n)$ for all $n\ge1$.
\end{lem}

\begin{proof}
Since both~$h_1$ and~$\tau$ are multiplicative, it suffices to observe that $|h_1(p^j)| \le j+1 = \tau(p^j)$ for all prime powers~$p^j$.
\end{proof}

\begin{lem}\label{lem:h1pre}
For all $k\ge1$, $\sum_{n \leq x} |h_k(n)|\ll x^{1/k}\log x$.
\end{lem}

\begin{proof}
Note that $h_k(n)=h_1(m)$ if $n=m^k$ is a perfect $k$th power and otherwise $h_k(n)=0$. It follows that $\sum_{n \leq x}|h_k(n)|=\sum_{n \leq x^{1/k}} |h_1(n)|$; thus it suffices to prove the lemma for $k=1$. But since $|h_1(n)| \le \tau(n)$ for all $n\ge1$, the estimate $\sum_{n \leq x} |h_1(n)|\ll x\log x$ follows immediately from the classical evaluation of $\sum_{n\le x} \tau(n)$ (see for example~\cite[Theorem~2.3]{MV07}).
\end{proof}

\begin{lem}\label{lem:h1}
There exists an absolute constant~$c>0$ such that $H_k(x) \ll x^{1/k}\exp\bigl(-c\frac{(\log x)^{3/5}}{(\log \log x)^{1/5}} \bigr)$ for all $k\ge1$.
\end{lem}

\begin{proof}
We have $H_k(x) = H_1(x^{1/k})$ by the same reasoning as in the proof of Lemma~\ref{lem:h1pre}, and so it suffices to prove the lemma for $k=1$. We leverage a known result for the Mertens function $M(x) = \sum_{n\le x} \mu(n)$, which uses contour integration to show that
\begin{equation}\label{eq:M(x)}
M(x) \ll x\exp\biggl(-c\frac{(\log x)^{3/5}}{(\log \log x)^{1/5}} \biggr)    
\end{equation}
for some absolute positive constant~$c$ (see for example~\cite[Theorem~12.7]{I85}). It is possible to slightly modify that proof to deal with $H_1(x)$; instead, however, we give an alternative derivation that uses the result~\eqref{eq:M(x)} directly rather than modifying its proof.

Since~$h_1$ is the Dirichlet convolution of~$\mu$ with itself, the hyperbola method implies that
\begin{align}
H_1(x)
&= \sum_{n \leq \sqrt{x}} \mu(n) M\biggl(\frac{x}{n}\biggr) + \sum_{m \leq \sqrt{x}} \mu(m) M\biggl(\frac{x}{m}\biggr) - M(\sqrt{x})M(\sqrt{x}) \notag \\
&\ll \sum_{n \leq \sqrt{x}}  \biggl|M\biggl(\frac{x}{n}\biggr)\biggr| + \bigl| M(\sqrt{x}) \bigr|^2 \label{eq:h(n)}.
\end{align}
Inequalities~\eqref{eq:M(x)} and~\eqref{eq:h(n)} imply that there are absolute constants $0<c''<c'<c$ such that
\begin{align*}
H_1(x)
&\ll\sum_{n \leq \sqrt{x}} \frac{x}{n}\exp\biggl(-c\frac{(\log (x/n))^{3/5}}{(\log \log (x/n))^{1/5}} \biggr) + x \exp\biggl(-2c\frac{(\log \sqrt{x})^{3/5}}{(\log \log \sqrt{x})^{1/5}} \biggr)\\
&\ll \sum_{n \leq \sqrt{x}} \frac{x}{n}\exp\biggl(-c'\frac{(\log x)^{3/5}}{(\log \log x)^{1/5}} \biggr) \\
&\ll x\log x\cdot \exp\biggl(-c'
\frac{(\log x)^{3/5}}{(\log \log x)^{1/5}} \biggr) \ll x\exp\biggl(-c''\frac{(\log x)^{3/5}}{(\log \log x)^{1/5}} \biggr),
\end{align*}
as desired.
\end{proof}

We now have the tools we need to establish an unconditional upper bound for the error term associated with the function~$g_k$ from Definition~\ref{k and g def}.

\begin{prop}\label{prop:conv-1}
There exists an absolute constant~$c>0$ such that $$E_{g_k}(x)\ll x^{1/k}\exp\biggl(-c\frac{(\log x)^{3/5}}{(\log \log x)^{1/5}} \biggr)$$ for all $k\ge2$.
\end{prop}

\begin{proof}
Since $D_{g_k}(s) = \zeta(s)\zeta(ks)^{-2}$, we have $G_{g_k}(x)={x}/{\zeta(k)^2}$. Since~$g_k$ is the Dirichlet convolution of~$h_k$ and~$1$, we can apply the hyperbola method, with a parameter $z \in (\sqrt{x},x)$ to be determined and with $y=\frac xz$:
\begin{align*}
F_{g_k}(x) = \sum_{mn \leq x} h_k(n) &= \sum_{n \leq z} h_k(n) \biggl\lfloor \frac{x}{n} \biggr\rfloor + \sum_{m\leq y} H_k\biggl(\frac{x}{m}\biggr) - \lfloor y \rfloor H_k(z) \\
&= x \sum_{n\le z} \frac{h_k(n)}n + O \biggl( \sum_{n\le z} |h_k(n)| + \sum_{m\le y} \biggl| H_k\biggl(\frac{x}{m}\biggr) \biggr| + y |H_k(z)| \biggr).
\end{align*}
Since $\sum_{n=1}^{\infty} h_k(n)/n = 1/\zeta(k)^2$, we use Lemma~\ref{lem:h1pre} to conclude that
\begin{equation} \label{eq:G1}
E_{g_k}(x) = F_{g_k}(x) - G_{g_k}(x) \ll x \biggl| \sum_{n>z} \frac{h_k(n)}n \biggr| + z^{1/k}\log z + \sum_{m\le y} \biggl| H_k\biggl(\frac{x}{m}\biggr) \biggr| + y |H_k(z)|.
\end{equation}
By Lemma~\ref{lem:h1} and partial summation,
\begin{equation}\label{eq:G2}
\frac{1}{\zeta(k)^2}-\sum_{n\leq z} \frac{h_k(n)}{n}\ll \frac{|H_k(z)|}{z}+\int_{z}^\infty \frac{|H_k(t)|}{t^2} \,dt \ll z^{1/k-1} \exp\biggl(-c\frac{(\log z)^{3/5}}{(\log \log z)^{1/5}} \biggr).
\end{equation}
Since $y \leq \sqrt{x}$, we have $x/m \geq \sqrt{x}$ for each $m \leq y$. Thus, by Lemma~\ref{lem:h1}, there is a constant $c'\in (0,c)$ such that
\begin{align}
\sum_{m\leq y} \biggl(H_k\biggl(\frac{x}{m}\biggr)-H_k(z)\biggr) 
&\ll \sum_{m \leq y} \biggl(\frac{x}{m}\biggr)^{1/k} \exp\biggl(-c\frac{(\log (x/m))^{3/5}}{(\log \log (x/m))^{1/5}} \biggr)\notag\\
&\ll x^{1/k}\sum_{m \leq y} m^{-1/k} \exp\biggl(-c'\frac{(\log x)^{3/5}}{(\log \log x)^{1/5}} \biggr)\notag\\
&\ll \frac{x}{z^{1-1/k}}\exp\biggl(-c'\frac{(\log x)^{3/5}}{(\log \log x)^{1/5}} \biggr).\label{eq:G3}
\end{align}
Combining inequalities~\eqref{eq:G1},~\eqref{eq:G2}, and~\eqref{eq:G3}, we conclude that
\begin{align}    
E_g(x) &\ll \frac{x}{z^{1-1/k}}\exp\biggl(-c'\frac{(\log x)^{3/5}}{(\log \log x)^{1/5}} \biggr)+\sum_{n \leq z}|h_k(n)|\notag\\
&\ll \frac{x}{z^{1-1/k}}\exp\biggl(-c'\frac{(\log x)^{3/5}}{(\log \log x)^{1/5}} \biggr)+z^{1/k}\log x. \label{eq:G4} 
\end{align}
If we set
$$
z=x \exp\biggl(-c'\frac{(\log x)^{3/5}}{(\log \log x)^{1/5}} \biggr),
$$
then we conclude that there exists $c'' \in (0,c')$ such that $E_g(x)\ll x^{1/k}\exp\bigl(-c''\frac{(\log x)^{3/5}}{(\log \log x)^{1/5}} \bigr)$.
\end{proof}

If we assume RH, then we can strengthen the above bound on $E_{g_k}(x)$. To do so, we first need a conditional estimate on the tail of the Dirichlet series for $1/\zeta(s)^2$, which we express in the following lemma using the function~$h_1$ from Definition~\ref{k and g def}.

\begin{lem}\label{lem:poly}
Assume RH. Let $\ep>0$ and $\sigma_0\geq \frac{1}{2}+\ep$. Uniformly for $\frac{1}{2}+\ep\leq \sigma \leq \sigma_0$, 
$$
\frac{1}{\zeta(s)^2}-\sum_{n\leq x} h_1(n)n^{-s}\ll_{\ep, \sigma_0} x^{1/2-\sigma+\ep}.   
$$    
\end{lem}

\begin{proof}
Recall that $\sum_{n=1}^{\infty} h_1(n)n^{-s}=\zeta(s)^{-2}$ for $\sigma>1$. Assume that $\frac{1}{2}+\ep\leq \sigma \leq \sigma_0$. By Lemma~\ref{lem:h1prepre}, we have $|h_1(n)n^{-s}|<\tau(n)n^{-1/2}\ll n^{-1/2+\ep}$. Thus, a truncated Perron's formula~\cite[Corollary 5.3]{MV07} implies that
\begin{equation}\label{eq:partial}
\sum_{n\leq x} h_1(n)n^{-s}=\frac{1}{2\pi i}\int_{2-iT}^{2+iT} \frac{1}{\zeta(s+w)^2}\frac{x^w}{w} \,dw +R(x),        
\end{equation}
where
\begin{equation}\label{eq:partial3}
R(x) \ll\sum_{\substack{x/2<n<2x \\n \neq x}} |h_1(n)n^{-s}| \min \biggl(1, \frac{x}{T|x-n|} \biggr)+ \frac{4^{2}+x^{2}}{T} \sum_{n=1}^{\infty} \frac{|h_1(n)n^{-s}|}{n^{2}}\ll \frac{x^2}{T}
\end{equation}
(we may assume that~$x$ is an integer).
Let $\ep'=\ep/(\sigma_0+\frac52)$. We shift the contour integral in equation~\eqref{eq:partial} leftwards from $\Re(w)=2$ to $\Re(w)=\frac{1}{2}-\sigma+\ep'$, noting that the only pole of ${x^w}/{w}{\zeta(s+w)^2}$ inside the contour is the simple pole at $w=0$. Thus
\begin{multline}
\frac{1}{2\pi i}\int_{2-iT}^{2+iT} \frac{1}{\zeta(s+w)^2}\frac{x^w}{w} \,dw \\
=\frac{1}{\zeta(s)^2}+\frac{1}{2\pi i}\bigg(\int_{\frac{1}{2}-\sigma+\ep'+iT}^{2+iT} + \int_{\frac{1}{2}-\sigma+\ep'-iT}^{\frac{1}{2}-\sigma+\ep'+iT} + \int_{2-iT}^{\frac{1}{2}-\sigma+\ep'-iT}\bigg) \frac{1}{\zeta(s+w)^2}\frac{x^w}{w} \,dw.
\label{eq:partial2}
\end{multline}
Since $\zeta(z)^{-1}\ll_{\ep'} T^{\ep'}$ holds uniformly for $\frac{1}{2}+\ep' \leq \Re(z) \leq 2$ and $|\Im(z)|\leq T$ (see~\cite[Theorem~13.23]{MV07}), it follows that
\begin{align*}
\int_{\frac{1}{2}-\sigma+\ep'-iT}^{\frac{1}{2}-\sigma+\ep'+iT} \frac{1}{\zeta(s+w)^2}\frac{x^w}{w} \,dw &\ll_{\ep'} x^{1/2-\sigma+\ep}T^{\ep'} \\
\int_{\frac{1}{2}-\sigma+\ep' \pm iT}^{2 \pm iT} \frac{1}{\zeta(s+w)^2}\frac{x^w}{w} \,dw &\ll_{\ep'} \sigma T^{\ep'} \cdot \frac{x^2}{T} \ll_{\ep', \sigma_0} \frac{x^{2}}{T^{1-\ep'}}.
\end{align*}
Combining these estimates with equations~\eqref{eq:partial}--\eqref{eq:partial2}, 
and setting $T=x^{3/2+\sigma}$,
we conclude that
$$
\frac{1}{\zeta(s)^2}-\sum_{n<x} h_1(n)n^{-s}\ll_{\ep',\sigma_0} x^{1/2-\sigma+\ep'} T^{\ep'}+\frac{x^{2}}{T^{1-\ep'}}+\frac{x^2}{T} \ll_{\ep',\sigma_0} x^{1/2-\sigma+{} \ep'(\sigma+5/2)}{}\leq x^{1/2-\sigma+ \ep}
$$
as required.
\end{proof}

\begin{prop}\label{prop:conv-1RH}
Let $k \geq 2$. Assuming RH, $E_{g_k}(x)\ll_{\ep} x^{1/(k+1)+\ep}$ for each $\ep>0$.
\end{prop}

\begin{proof}
%Since $D_{g_k}(s) = \zeta(s)\zeta(ks)^{-2}$, we have $G_{g_k}(x)={x}/{\zeta(k)^2}$. %Note that $(g_k(n))$ is the Dirichlet convolution of~$(h_k(n))$ and~$(1)$. 
We adapt the proof that Montgomery and Vaughan~\cite{MV81} used for the indicator function of $k$-free numbers. Recall that $h_k(n)=h_1(m)$ if $n=m^k$ is a perfect $k$th power and $h_k(n)=0$ otherwise. Let $y=x^{1/(k+1)}$, and define
$$
A(s)=\frac{1}{\zeta(s)^2}-\sum_{n \leq y} h_1(n)n^{-s}.
$$
If we set $\widetilde{h}(n)=h_k(n)$ for $n>y^k$ and $\widetilde{h}(n)=0$ otherwise, we see that the Dirichlet series for~$\widetilde{h}$ is precisely~$A(ks)$.
Define~$\widetilde{g}$ to be the Dirichlet convolution of~$\widetilde{h}$ and~$1$, so that the Dirichlet series for~$\widetilde{g}$ is $\zeta(s)A(ks)$.
Also note that $|\widetilde{g}(n)|\leq \sum_{d\mid n}|\widetilde{h}(d)|\leq \sum_{d\mid n}\tau(d)\ll_{\ep} n^\ep$ for each $\ep>0$ by Lemma~\ref{lem:h1prepre}. Using the truncated Perron's formula~\cite[Corollary 5.3]{MV07} with the choice $T=x$,
\begin{equation}\label{eq:S2}
\sum_{n \leq x} \widetilde{g}(n)=\frac{1}{2\pi i}\int_{c-ix}^{c+ix} \zeta(s)A(ks) \frac{x^s}{s} \,ds +R(x).        
\end{equation}
where $c=1+1/(k+1)$ and
\begin{equation}\label{eq:S6}
R(x) \ll\sum_{\substack{x/2<n<2x \\n \neq x}} |\widetilde{g}(n)| \min \biggl(1, \frac{x}{x|x-n|} \biggr)+ \frac{4^{c}+x^{c}}{x} \sum_{n=1}^{\infty} \frac{|\widetilde{g}(n)|}{n^{c}}\ll x^\ep + x^{1/(k+1)}.
\end{equation}
We shift the contour leftwards from $\sigma=c$ to $\sigma=\frac{1}{2}$ and notice that the only pole of $\zeta(s)A(ks)$ inside the contour is the simple pole at $s=1$. Thus,
\begin{align}\label{eq:S3}
&\frac{1}{2\pi i}\int_{c-ix}^{c+ix} \zeta(s)A(ks) \frac{x^s}{s}=xA(k)+\frac{1}{2\pi i}\bigg(\int_{\frac{1}{2}-ix}^{\frac{1}{2}+ix}+\int_{\frac{1}{2}+ix}^{c+ix}+\int_{c-ix}^{\frac{1}{2}-ix}\bigg) \zeta(s)A(ks) \frac{x^s}{s} \,ds.
\end{align}
By \cite[Theorem 13.18]{MV07}, $\zeta(s)\ll_{\ep} x^\ep$ uniformly for $\sigma\geq \frac{1}{2}$ and $1\leq |t|\leq x$. By Lemma~\ref{lem:poly}, $A(s)\ll_{\ep} y^{1/2-\sigma+\ep}$ holds uniformly for $\frac{1}{2}+\ep \leq \sigma \leq k+2$, and thus $A(ks)\ll_{\ep} y^{1/2-k\sigma+\ep}$ holds uniformly for $\frac{1}{2}+\ep \leq \sigma \leq c$. Therefore
\begin{align*}
\int_{\frac{1}{2}-ix}^{\frac{1}{2}+ix} \zeta(s)A(ks) \frac{x^s}{s} &\ll_{\ep} x^{1/2+\ep} y^{(1-k)/2+\ep}, \\
\int_{\frac{1}{2}\pm ix}^{c\pm ix} \zeta(s)A(ks) \frac{x^s}{s} &\ll_{\ep} x^{c-1+\ep} y^{(1-k)/2+\ep}.
\end{align*}
Combining these estimates with equations~\eqref{eq:S2}--\eqref{eq:S3}, we conclude that
\begin{equation}\label{eq:S5}
\sum_{n \leq x} \widetilde{g}(n)=xA(k)+O_{\ep}\big(x^{1/(k+1)}+x^{1/2+\ep} y^{(1-k)/2+\ep}\big).      
\end{equation}

On the other hand,
\begin{align*}
\sum_{mn \leq x} \bigl( h_k(n)-\widetilde{h}(n) \bigr)
&=\sum_{\substack{mn \leq x\\n\leq y^k}} h_k(n)
=\sum_{\substack{mn^k \leq x\\n \leq y}} h_1(n)
=\sum_{n \leq y} h_1(n) \biggl\lfloor \frac{x}{n^k} \biggr\rfloor \\
&= x \sum_{n \leq y} \frac{h_1(n)}{n^k}+O \biggl( \sum_{n\le y} |h_1(n)| \biggr)
=x \biggl( \frac{1}{\zeta(k)^2} -A(k)\biggr) + O(y\log y)
\end{align*}
by Lemmas~\ref{lem:poly} and~\ref{lem:h1pre}.
From this estimate and equation~\eqref{eq:S5}, it follows that
\begin{align*}
F_{g_k}(x) 
&= \sum_{n \leq x} g_k(n)=\sum_{mn \leq x} h_k(n)\\
&=\sum_{mn \leq x} \big(h_k(n)-\widetilde{h}(n)\big)+\sum_{mn \leq x} \widetilde{h}(n) \\
&=\sum_{mn \leq x} \big(h_k(n)-\widetilde{h}(n)\big)+\sum_{n \leq x}\widetilde{g}(n) \\
&=\frac{x}{\zeta(k)^2}+O_{\ep}\bigl( y\log y+ x^{1/(k+1)}+ x^{1/2+\ep} y^{(1-k)/2+\ep}\bigr).  
\end{align*}
Since $y=x^{1/(k+1)}$, this error term is $O_{\ep}\bigl( x^{1/(k+1)+\ep} \bigr)$
and therefore
$$
E_{g_k}(x) = F_{g_k}(x) - G_{g_k}(x) \ll_{\ep} x^{1/(k+1)+\ep}
$$
as required.
\end{proof}

\subsection{The convolution method}\label{sec:conv}

In this section, we apply the convolution method to prove both parts of Theorem~\ref{thm:unconditional} as well as Theorem~\ref{thm:ubRH}(b). The following key lemma is inspired by~\cite[Lemma~1]{BG58}.

\begin{lem}\label{lem:conv}
 Let
$Y(s) = \sum_{n=0}^{\infty} y_n n^{-s}$ and $Z(s) = \sum_{n=0}^{\infty} z_n n^{-s}$ be two Dirichlet series; assume that $Y(s)$ converges absolutely for $\sigma>1$ and that~$k$ is a positive integer such that $Z(s)$ converges absolutely for $\sigma>\frac{1}{k+1}$. Suppose that the partial sum $T(x) = \sum_{n \leq x} y_n$ has the form 
\begin{equation} \label{all simple}
T(x) = \sum_{j=1}^{k} \Res\biggl(Y(s) \cdot \frac{x^s}{s}, \frac{1}{j} \biggr)+R(x) = \sum_{j=1}^k r_j x^{1/j}+R(x).
\end{equation}
Assume further that there is $f \in \cF$ such that
$D_f(s) =Y(s)Z(s)$ for $\sigma>1$. 
\begin{enumerate}
\item If $R(x) \ll 1$, then $E_f(x) \ll_{\ep} x^{1/(k+1)+\ep}$ for every $\ep>0$.
\item If $\theta>\frac{1}{k+1}$ is a real number, and $\widetilde{R}(x)$ is an eventually increasing function that satisfies $R(x)\ll x^{\theta}\widetilde{R}(x)$, then $E_f(x) \ll x^{\theta} \widetilde{R}(x)$.
\end{enumerate}
\end{lem}

\begin{rem}
In addition to the dependence on~$\ep$ in part~(a), the implied constants in the above estimates on $E_f(x)$ depend on $Y(s)$ and $Z(s)$, as well as on~$\theta$ and on the implied constant in the assumed upper bound for $R(x)$ in part~(b). In the applications below, these ancillary quantities will all be chosen in terms of the function $f\in \cF$; consequently, the implied constants will simply depend on~$f$ in those cases.
\end{rem}

\begin{rem}
In part~(b), we see that the deduced upper bound for $E_f(x)$ is the same as the assumed upper bound for $R(x)$ as long as that bound is sufficiently ``nice'' with respect to~$\theta$. We will have this sort of nice upper bound in all applications of Lemma~\ref{lem:conv}(b) below.
\end{rem}

\begin{proof}[Proof of Lemma~\ref{lem:conv}]
Let $1\leq j \leq k$ and assume that $D_f(s)$ has a pole at $\frac1j$. Since $Z(s)$ converges absolutely for $\sigma>\frac{1}{k+1}$ and $D_f(s) =Y(s)Z(s)$, it follows that $\frac1j$ is a pole of $Y(s)$, which must be simple since the residues in equation~\eqref{all simple} have no logarithmic terms. We deduce that
\begin{align}
\Res\biggl(D_f(s) \cdot \frac{x^s}{s}, \frac{1}{j} \biggr)
&=Z\biggl( \frac{1}{j} \biggr)\Res\biggl(Y(s) \cdot \frac{x^s}{s}, \frac{1}{j} \biggr) =Z\biggl( \frac{1}{j} \biggr)r_jx^{1/j}. \label{eq:res}
\end{align}
By partial summation, it follows that for any $0<\ep<\frac1j-\frac1{k+1}$,
\begin{equation}\label{eq:uberror}
Z\biggl( \frac{1}{j} \biggr)-\sum_{n\leq x}z_n n^{^{-{1}/{j}}} \ll \sum_{n>x} |z_n|n^{^{-{1}/{j}}} \le x^{{1}/(k+1)-{1}/{j}+\ep} \sum_{n>x} |z_n|n^{-({1}/(k+1)+\ep)} \ll_{\ep} x^{{1}/(k+1)-{1}/{j}+\ep}.
\end{equation}
Since $(f(n))$ is the Dirichlet convolution of $(y_n)$ and $(z_n)$,
\begin{equation}\label{eq:F}
F_f(x) = \sum_{n \leq x} f(n) = \sum_{n \leq x} z_nT\biggl( \frac xn \biggr) = \sum_{j=1}^k \biggl( \sum_{n\leq x}z_n n^{-1/j} \biggr)r_j x^{1/j}+ \sum_{n \leq x} z_n R\biggl( \frac xn \biggr).
\end{equation}
In light of the definition~\eqref{eq:G} of $G_f(x)$, equations~\eqref{eq:res}--\eqref{eq:F} together imply that
$$
E_f(x) =F_f(x)-G_f(x)\ll_{\ep} x^{1/(k+1)+\ep}+\sum_{n \leq x} |z_n| \biggl| R\biggl( \frac xn \biggr) \biggr|.
$$
This bound implies both parts of the lemma:
\begin{enumerate}
\item If $R(x)\ll 1$, then $\sum_{n \leq x} |z_n| |R(\frac xn)| \ll \sum_{n \leq x} |z_n|\ll x^{1/(k+1)+\ep}$ by partial summation, and thus $E_f(x) \ll_{\ep} x^{1/(k+1)+\ep}$ as required. 
\item Since $\theta>\frac{1}{k+1}$, it follows from the assumptions that
$$
\sum_{n \leq x} |z_n| \biggl| R\biggl( \frac xn \biggr) \biggr| \ll x^{\theta}\widetilde{R}(x) \sum_{n \leq x} \frac{|z_n|}{n^{\theta}} \ll x^{\theta}\widetilde{R}(x).
$$
Choosing $\ep<\theta-1/(k+1)$, we conclude that $E_f(x) \ll_{\ep} x^{1/(k+1)+\ep}+x^{\theta}\widetilde{R}(x) \ll x^{\theta}\widetilde{R}(x)$, as required.
\qedhere
\end{enumerate}
\end{proof}

Next, we give the proof of Theorem~\ref{thm:unconditional} on unconditional upper bounds on $E_f(x)$, the first half of which is extremely short.

\begin{proof}[Proof of Theorem~\ref{thm:unconditional}(a)]
Let $(y_n)$ be the indicator function of~$k$th powers, whose Dirichlet series is $Y(s) = \sum_{n=0}^{\infty} y_n n^{-s}= \zeta(ks)$ and whose summatory function is $\sum_{n \leq x} y_n=x^{1/k}+O(1)$. By Theorem~\ref{thm:algorithm}, we can write $D_f(s) =Y(s)Z(s)$, where $Z(s)$ converges absolutely for $\sigma>\frac{1}{k+1}$. Lemma~\ref{lem:conv}(a) then implies that $E_f(x) \ll_{\ep} x^{1/(k+1)+\ep}$ for each $\ep>0$.
\end{proof}

\begin{proof}[Proof of Theorem~\ref{thm:unconditional}(b)]
We divide our discussion into two cases; in each case, we will eventually apply Lemma~\ref{lem:conv}(b) with the choices
\begin{equation}\label{thetaR}
\theta \in \biggl( \frac{1}{k+1}, \frac{1}k \biggr) \quad\text{and}\quad
\widetilde{R}(x) =x^{{1}/{k}-\theta}\exp\biggl(-c\frac{(\log x)^{3/5}}{(\log \log x)^{1/5}} \biggr).
\end{equation}

\begin{enumerate}[(i)]
\item We first consider the case that~$f$ is of M\"obius-type. Let $Y(s) = \zeta(ks)^{-1}$. By Proposition~\ref{prop:typeI}, we can write $D_f(s) =Y(s)Z(s)$ for $\sigma>1$, where $Z(s)$ is a Dirichlet series that converges absolutely for $\sigma>\frac{1}{k+1}$. Let $(y_n)$ be the sequence supported on perfect~$k$th powers such that $y_{m^k}= \mu(m)$ for each nonnegative integer~$m$, so that $Y(s) = \sum_{n=0}^{\infty} y_n n^{-s}$. Note that $T(x) = \sum_{n \leq x} y_n= \sum_{m \leq x^{1/k}} \mu(m) =M(x^{1/k})$ in terms of the Mertens function. From the best known error term on~$M(x)$ (see for example~\cite[Theorem~12.7]{I85}), we see that $T(x)\ll R(x)$, where $R(x)$ is given by the right-hand side of the estimate~\eqref{eq:boundpowerfree}. Then the function $\widetilde{R}(x)$ given in equation~\eqref{thetaR} is eventually increasing and satisfies $T(x)\ll x^\theta\widetilde{R}(x)$, and so Lemma~\ref{lem:conv} implies the required upper bound on~$E_f(x)$.

\item Next we consider the case that~$f$ is of powerfree-type. By Proposition~\ref{prop:typeII}, $k$ is the smallest integer such that $\ep_k \neq 1$. If $\ep_k=0$, let $(y_n)$ be the indicator function of the~$k$-free numbers; if $\ep_k=-1$, let $y_n=g_k(n)$, where $g_k$ is defined in Definition~\ref{k and g def}. In both cases, the Dirichlet series of $(y_n)$ is $Y(s)=\sum_{n=1}^{\infty} y_n n^{-s}=\zeta(s)/\zeta(ks)^{1+|\ep_k|}$, and Proposition~\ref{prop:typeII} implies that we can write $D_f(s) =Y(s)Z(s)$ for $\sigma>1$, where $Z(s)$ converges absolutely for $\sigma>\frac{1}{k+1}$. Let $T(x) = \sum_{n \leq x} y_n= x\Res(Y(s)x^s/s,1)+R(x)$. Then the upper bound on $R(x)$ is exactly the right-hand side of the estimate~\eqref{eq:boundpowerfree}: this is from a result of Walfisz~\cite[Section~5.6]{W63} when $\ep_k=0$, and follows from Proposition~\ref{prop:conv-1} when $\ep_k=-1$. The theorem in this case follows from Lemma~\ref{lem:conv} by setting $\theta$ and $\widetilde{R}(x)$ as in equation~\eqref{thetaR}. \qedhere
\end{enumerate}
\end{proof}

We also present the proof of Theorem~\ref{thm:ubRH}(b) on upper bounds on $E_f(x)$ for $f$ of powerfree-type under RH. The proof is very similar to the proof of Theorem~\ref{thm:unconditional}(b) that just concluded.

\begin{proof}[Proof of Theorem~\ref{thm:ubRH}(b)]
As in case~(ii) in the proof of Theorem~\ref{thm:unconditional}, we let $(y_n)$ be the indicator function of~$k$-free numbers when $\ep_k=0$, and let $y_n=g_k(n)$ when $\ep_k=-1$. We can then write $D_f(s)=Y(s)Z(s)$ for $\sigma>1$, where $Y(s)$ is the Dirichlet series of the sequence $(y_n)$ and $Z(s)$ converges absolutely for $\sigma>\frac{1}{k+1}$.
We then have $T(x) = \sum_{n \leq x} y_n=x/\zeta(k)^{1+|\ep_k|}+O_{\ep}(x^{1/(k+1)+\ep})$ for each $\ep>0$: this is a result by Montgomery and Vaughan~\cite{MV81} when $\ep_k=0$, and follows from Proposition~\ref{prop:conv-1RH} when $\ep_k=-1$. 
In both cases, we have $x/\zeta(k)^{1+|\ep_k|} = \Res(Y(s)x^s/s, 1)$ and that $s=1$ is the only real pole of $Y(s)$. Thus, by setting $\theta=\frac1{k+1}+\ep$ and $\widetilde{R}(x) =1$, Lemma~\ref{lem:conv} implies that $E_f(x)\ll_{\ep} x^{1/(k+1)+\ep}$.
\end{proof}

\subsection{The M\"obius-type and powerfull-type cases under RH} \label{sec:conditional2}

In this section, we prove Theorem~\ref{thm:ubRH}(a) via the method of contour integration.  We need the following upper bounds on $\zeta(s)$ and $1/\zeta(s)$. 

\begin{lem}\label{lem:Tep}
Assume RH. 
\begin{enumerate}
    \item For each $\delta>0$, the bounds $\zeta(s) \ll_{\delta} 1$ and $\zeta(s)^{-1} \ll_{\delta} 1$ hold uniformly in the compact region $\{s\colon \frac{1}{2}\leq \sigma \leq 1-\delta,\, |t|\leq 13\}$.
    \item There is a positive constant $C$ such that both
$$
|\zeta(s)| \leq \exp\biggl(C\frac{\log x}{\log \log x}\biggr)
\quad\text{and}\quad
|\zeta(s)^{-1}| \leq \exp\biggl(C\frac{\log x}{\log \log x}\biggr)
$$
hold for $\sigma \geq \frac{1}{2}+\frac{1}{\log \log x}$ and $1700\leq |t|\leq x-4$. 
\end{enumerate}
\end{lem}

\begin{proof}
Part~(a) is immediate: since all nontrivial zeros $\rho$ of $\zeta$ have imaginary part at least~$14$, both $\zeta(s)$ and $\zeta(s)^{-1}$ are continuous in the compact region $\{s\colon \frac{1}{2}\leq \sigma \leq 1-\delta, |t|\leq 13\}$, and the statement follows (indeed without even invoking RH).

As for part~(b), let $x\geq \exp(\exp(10))$ be fixed. Consider the two functions
$$
g(y)=\frac{\log y}{\log \log y}
\quad\text{and}\quad
h(y)=\frac{\log y}{\log \log y} \log \frac{e \log \log x}{\log \log y}.
$$
Note that $\log g(y)=\log \log y-\log \log \log y$ and so
\[
\frac{g'(y)}{g(y)}=\frac{1}{y\log y}-\frac{1}{y\log y \log \log y}=\frac{1}{y\log y} \biggl(1-\frac{1}{\log \log y}\biggr)>0
\]
when $y \geq e^e$. Thus $g(y)$ is increasing when $y \geq 16$, and in particular $g(y)\leq g(x)$ when $16 \leq y \leq x$. Similarly, 
$$\frac{h'(y)}{h(y)}=\frac{1}{y\log y}-\frac{1}{y\log y \log \log y}-\frac{1}{y\log y} \frac{1}{\log \log y} \cdot \log \frac{e\log \log x}{\log \log y}. $$
When $1700\leq y \leq x$, we have
$$
\frac{1}{\log \log y}+\frac{1}{\log \log y \cdot \log \frac{e\log \log x}{\log \log y}}\leq \frac{2}{\log \log y}<0.997<1;
$$
so $h(y)$ is increasing when $1700\leq y \leq x$, and in particular $h(y)\leq h(x)=g(x)$ when $1700\leq y \leq x$.

Write $\tau=|t|+4$ so that bounds involving $\log\tau$ are valid even when~$t$ is small. By~\cite[Theorem~13.18]{MV07}, there is an absolute positive constant $C$ such that
$$
|\zeta(s)| \leq \exp\biggl(C\frac{\log \tau}{\log \log \tau}\biggr)=\exp(Cg(\tau))
$$
for $\sigma \geq \frac{1}{2}$ and $|t|\geq 1$. Thus, when $\sigma \geq \frac{1}{2}$ and $1\leq |t|\leq x-4$, we have $|\zeta(s)|\leq \exp(Cg(\tau))\leq \exp(Cg(x))$. 

To bound $|1/\zeta(s)|$, we consider two cases. When $\sigma \geq \frac{1}{2}+\frac{1}{\log \log \tau}$, we apply a similar argument as above by using the first part of \cite[Theorem~13.23]{MV07}. Next assume $\frac{1}{2}<\sigma \leq \frac{1}{2}+\frac{1}{\log \log \tau}$ and $1700\leq |t|\leq x-4$. In this case, by the second part of \cite[Theorem~13.23]{MV07}, there is an absolute positive constant $C$ such that
$$
|1/\zeta(s)| \leq \exp(Ch(\tau))\leq \exp(Ch(x))=\exp(Cg(x)),
$$
as required.
\end{proof}

We conclude the paper with a proof of Theorem~\ref{thm:ubRH}(a).
\begin{proof}[Proof of Theorem~\ref{thm:ubRH}(a)]
Let $f \in \cF$ be fixed and let~$k$ be its initial index.
By the discussion in Section~\ref{sec:omgea}, for $\sigma>1$, we can write
\begin{equation}\label{eq:factor}
D_f(s) =W(s) \cdot \prod_{j=k}^{2k-1} \zeta(js)^{a_j},
\end{equation}
where $a_j$ are integers and $W(s)$ converges absolutely for $\sigma>\frac{1}{2k}$. 

Let $x>\exp(\exp(10))$ be sufficiently large and set
$$
\sigma_1=\frac{1}{k} \biggl(\frac{1}{2}+\frac{1}{\log \log x}\biggr), \quad T=x^{1-\sigma_1},
$$
so that $k \sigma_1=\frac{1}{2}+\frac{1}{\log \log x}$; note that $(2k-1)\sigma_1<1$ and $2k(T+4)<x$. Let $C_0$ be the positive constant from Lemma~\ref{lem:Tep}(b). Let $B=\sum_{j=k}^{2k-1} |a_j|$ and let $C_1=BC_0$. 

Since $|f(n)|\leq 1$ for all $n \in \N$, we may assume without loss of generality that $x$ is a positive integer in the following discussion. Since $D_f(s)$ is analytic for $\sigma>1$, we set $\sigma_0=1+\frac{1}{\log x}$. By the truncated Perron's formula~\cite[Corollary 5.3]{MV07},
$$
F_f(x) = \sum_{n \leq x} f(n) = \frac{1}{2\pi i} \int_{\sigma_0-iT}^{\sigma_0+iT} D_f(s) \frac{x^s}{s} \,ds+ R(x)
$$
where
$$
R(x) \ll\sum_{\substack{x/2<n<2x \\n \neq x}} |f(n)| \min \biggl(1, \frac{x}{T|x-n|} \biggr)+ \frac{4^{\sigma_0}+x^{\sigma_0}}{T} \sum_{n=1}^{\infty} \frac{|f(n)|}{n^{\sigma_0}}.
$$
Therefore
$$
R(x)\ll \frac{x\log x}{T}+\frac{x \zeta(\sigma_0)}{T} \ll \frac{x\log x}{T}.
$$

We shift the contour leftwards from $\sigma= \sigma_0$ to $\sigma= \sigma_1$. Define the integrals
\begin{align*}
&
I_1(x) = \frac{1}{2\pi i} \int_{\sigma_0-iT}^{\sigma_0+iT} D_f(s) \frac{x^s}{s} \,ds,  \quad I_2(x) = \frac{1}{2\pi i} \int_{\sigma_0+iT}^{\sigma_1+iT} D_f(s) \frac{x^s}{s} \,ds, \\
& I_3(x) = \frac{1}{2\pi i} \int_{\sigma_1+iT}^{\sigma_1-iT} D_f(s) \frac{x^s}{s} \,ds,\quad 
I_4(x) = \frac{1}{2\pi i} \int_{\sigma_1-iT}^{\sigma_0-iT} D_f(s) \frac{x^s}{s} \,ds.
\end{align*}
Since $k\sigma_1>\frac{1}{2}$ and $W(s)$ is analytic for $\sigma>\sigma_1$, the only possible poles of the contour integral $I_1+I_2+I_3+I_4$ come from $s= \frac{1}{k}, \frac{1}{k+1}, \ldots, \frac{1}{2k-1}$. Thus, the residue theorem implies that
$$
I_1(x)+I_2(x)+I_3(x)+I_4(x) = \sum_{j=k}^{2k-1} \Res\biggl(D_f(s) \frac{x^s}{s}, \frac{1}{j} \biggr).
$$

We first bound $I_3(x)$. For $\sigma= \sigma_1$ and for each $k \leq j \leq 2k-1$, since $\frac{1}{2}+\frac{1}{\log \log x}\leq j \sigma_1 \leq (2k-1)\sigma<1$ and $j|t|+4<2k(T+4)<x$, it follows from Lemma~\ref{lem:Tep} that $\zeta(js)^{a_j} \ll \exp(|a_j|C_0\frac{\log x}{\log \log x})$. Thus, for $\sigma= \sigma_1$, we have 
$$
D_f(s) =W(s) \cdot \prod_{j=k}^{2k-1} \zeta(js)^{a_j} \ll \prod_{j=k}^{2k-1} \zeta(js)^{a_j} \ll \prod_{j=k}^{2k-1} \exp\biggl(|a_j|C_0\frac{\log x}{\log \log x}\biggr)= \exp\biggl(C_1\frac{\log x}{\log \log x}\biggr).
$$
It follows that
$$
I_3(x)\ll \int_{-T}^{T} \exp\biggl(C_1\frac{\log x}{\log \log x}\biggr) \frac{x^{\sigma_1}}{|\sigma_1+t|} \,dt \ll x^{\sigma_1} \exp\biggl(C_2\frac{\log x}{\log \log x}\biggr),
$$
where $C_2=C_1+1$.

Next we estimate $I_2(x)$. Let $k \leq j \leq 2k-1$ be fixed. When $\sigma \geq \sigma_1$, we have $j\sigma \geq k\sigma_1\geq \frac{1}{2}+\frac{1}{\log \log x}$. Thus, Lemma~\ref{lem:Tep} implies that
$$
\zeta\bigl(j(\sigma+iT)\bigr) \ll \exp\biggl(C_0\frac{\log x}{\log \log x}\biggr)
\quad\text{and}\quad
\zeta\bigl(j(\sigma+iT)\bigr)^{-1} \ll \exp\biggl(C_0\frac{\log x}{\log \log x}\biggr)
$$
uniformly for $\sigma\geq \sigma_1$, and thus that $D(\sigma+ix) \ll \exp(C_2\frac{\log x}{\log \log x})$ holds uniformly for $\sigma_1 \leq \sigma \leq \sigma_0$. Therefore
$$
I_2(x)\ll \int_{\sigma_1}^{\sigma_0} \exp\biggl(C_2\frac{\log x}{\log \log x}\biggr) \frac{x^{\sigma}}{T} \,d\sigma \ll \frac{x^{\sigma_0}}{x} \exp\biggl(C_2\frac{\log x}{\log \log x}\biggr)\ll \frac{x}{T} \exp\biggl(C_2\frac{\log x}{\log \log x}\biggr).
$$
A similar argument provides the same upper bound for $I_4(x)$.

We have shown that
\begin{align*}
F_f(x)-\sum_{j=k}^{2k-1} \Res\biggl(D_f(s) \frac{x^s}{s}, \frac{1}{j} \biggr) &\ll R(x)+I_2(x)+I_3(x)+I_4(x) \\
&\ll \frac{x\log x}{T}+\biggl(x^{\sigma_1}+\frac{x}{T}\biggr)\exp\biggl(C_2\frac{\log x}{\log \log x}\biggr).
\end{align*}
Since $T=x^{1-\sigma_1}$, it follows that
$$
F_f(x)-\sum_{j=k}^{2k-1} \Res\biggl(D_f(s) \frac{x^s}{s}, \frac{1}{j} \biggr) \ll x^{\sigma_1} \exp\biggl(C_2\frac{\log x}{\log \log x}\biggr) \ll x^{1/2k}\exp\biggl(C\frac{\log x}{\log \log x}\biggr),
$$
where $C=C_2+1$.

Finally, recall from equation~\eqref{eq:G} that
$$
G_f(x) = \sum_{j=k}^{2\ell} \Res\biggl(D_f(s) \frac{x^s}{s}, \frac{1}{j} \biggr) = \sum_{j=k}^{2k-1} \Res\biggl(D_f(s) \frac{x^s}{s}, \frac{1}{j} \biggr)+O(x^{1/2k}(\log x)^{O(1)}),
$$
from which we conclude the desired upper bound
\[
E_f(x) =F_f(x)-G_f(x)\ll x^{1/2k}\exp\biggl(C\frac{\log x}{\log \log x}\biggr).\qedhere
\]
\end{proof}

\section*{Acknowledgments}

The authors thank Fatma Çiçek, Michael J.~Mossinghoff, Jan-Christoph Schlage-Puchta, and Tim Trudgian for helpful discussions. The authors also thank the anonymous referees for their valuable comments and suggestions. The first author was supported in part by a Natural Sciences and Engineering Council of Canada (NSERC) Discovery Grant. The second author was supported in part by an NSERC fellowship.

\bibliographystyle{abbrv}
\bibliography{main}

\begin{thebibliography}{10}

\bibitem{A70}
T.~M. Apostol.
\newblock M\"{o}bius functions of order {$k$}.
\newblock {\em Pacific J. Math.}, 32:21--27, 1970.

\bibitem{BRS88}
R.~Balasubramanian, K.~Ramachandra, and M.~V. Subbarao.
\newblock On the error function in the asymptotic formula for the counting function of {$k$}-full numbers.
\newblock {\em Acta Arith.}, 50(2):107--118, 1988.

\bibitem{BFMT23}
D.~Banerjee, Y.~Fujisawa, T.~M. Minamide, and Y.~Tanigawa.
\newblock A note on the partial sum of {A}postol's {M}\"{o}bius function.
\newblock {\em Acta Math. Hungar.}, 170(2):635--644, 2023.

\bibitem{BG58}
P.~T. Bateman and E.~Grosswald.
\newblock On a theorem of {E}rd\"{o}s and {S}zekeres.
\newblock {\em Illinois J. Math.}, 2:88--98, 1958.

\bibitem{B01}
A.~Bege.
\newblock A generalization of {A}postol's {M}\"obius functions of order {$k$}.
\newblock {\em Publ. Math. Debrecen}, 58(3):293--301, 2001.

\bibitem{BL14}
H.~Bohr and E.~Landau.
\newblock Sur les z{\'e}ros de la fonction {{\(\zeta (s)\)}} de \emph{Riemann}.
\newblock {\em C. R. Acad. Sci., Paris}, 158:106--110, 1914.

\bibitem{D52}
G.~Dahlquist.
\newblock On the analytic continuation of {E}ulerian products.
\newblock {\em Ark. Mat.}, 1:533--554, 1952.

\bibitem{dSW}
M.~du~Sautoy and L.~Woodward.
\newblock {\em Zeta functions of groups and rings}, volume 1925 of {\em Lecture Notes in Mathematics}.
\newblock Springer-Verlag, Berlin, 2008.

\bibitem{EL31}
C.~J.~A. Evelyn and E.~H. Linfoot.
\newblock On a problem in the additive theory of numbers.
\newblock {\em Ann. of Math. (2)}, 32(2):261--270, 1931.

\bibitem{G93}
S.~M. Gonek.
\newblock An explicit formula of {L}andau and its applications to the theory of the zeta-function.
\newblock In {\em A tribute to {E}mil {G}rosswald: number theory and related analysis}, volume 143 of {\em Contemp. Math.}, pages 395--413. Amer. Math. Soc., Providence, RI, 1993.

\bibitem{GP89}
S.~W. Graham and J.~Pintz.
\newblock The distribution of {$r$}-free numbers.
\newblock {\em Acta Math. Hungar.}, 53(1-2):213--236, 1989.

\bibitem{H58}
C.~B. Haselgrove.
\newblock A disproof of a conjecture of {P}\'{o}lya.
\newblock {\em Mathematika}, 5:141--145, 1958.

\bibitem{H18}
G.~Hurst.
\newblock Computations of the {M}ertens function and improved bounds on the {M}ertens conjecture.
\newblock {\em Math. Comp.}, 87(310):1013--1028, 2018.

\bibitem{I42}
A.~E. Ingham.
\newblock On two conjectures in the theory of numbers.
\newblock {\em Amer. J. Math.}, 64:313--319, 1942.

\bibitem{I78}
A.~Ivi\'c.
\newblock On the asymptotic formulas for powerful numbers.
\newblock {\em Publ. Inst. Math. (Beograd) (N.S.)}, 23(37):85--94, 1978.

\bibitem{I85}
A.~Ivi\'c.
\newblock {\em The {R}iemann zeta-function}.
\newblock A Wiley-Interscience Publication. John Wiley \& Sons, Inc., New York, 1985.

\bibitem{L12}
E.~Landau.
\newblock \"{U}ber die {N}ullstellen der {Z}etafunktion.
\newblock {\em Math. Ann.}, 71(4):548--564, 1912.

\bibitem{L16}
H.-Q. Liu.
\newblock The distribution of squarefull integers ({II}).
\newblock {\em J. Number Theory}, 159:176--192, 2016.

\bibitem{MMT23}
G.~Martin, M.~J. Mossinghoff, and T.~S. Trudgian.
\newblock Fake mu's.
\newblock {\em Proc. Amer. Math. Soc.}, 151(8):3229--3244, 2023.

\bibitem{ABCPNT}
G.~Martin, P.~J.~S. Yang, A.~Bahrini, P.~Bajpai, K.~Benli, J.~Downey, Y.~Y. Li, X.~Liang, A.~Parvardi, R.~Simpson, E.~P. White, and C.~H. Yip.
\newblock An annotated bibliography for comparative prime number theory.
\newblock {\em Expo. Math.}, 43(3):Paper No. 125644, 124, 2025.

\bibitem{M17}
X.~Meng.
\newblock The distribution of {$k$}-free numbers and the derivative of the {R}iemann zeta-function.
\newblock {\em Math. Proc. Cambridge Philos. Soc.}, 162(2):293--317, 2017.

\bibitem{M97}
H.~Menzer.
\newblock On the distribution of powerful numbers.
\newblock {\em Abh. Math. Sem. Univ. Hamburg}, 67:221--237, 1997.

\bibitem{1897.Mertens}
F.~Mertens.
\newblock \"uber eine zahlentheoretische funktion.
\newblock {\em Sitzungsberichte Akad. Wien}, 106:761--830, 1897.

\bibitem{MV81}
H.~L. Montgomery and R.~C. Vaughan.
\newblock The distribution of squarefree numbers.
\newblock In {\em Recent progress in analytic number theory, {V}ol. 1 ({D}urham, 1979)}, pages 247--256. Academic Press, London-New York, 1981.

\bibitem{MV07}
H.~L. Montgomery and R.~C. Vaughan.
\newblock {\em Multiplicative number theory. {I}. {C}lassical theory}, volume~97 of {\em Cambridge Studies in Advanced Mathematics}.
\newblock Cambridge University Press, Cambridge, 2007.

\bibitem{M00}
P.~Moree.
\newblock Approximation of singular series and automata.
\newblock {\em Manuscripta Math.}, 101(3):385--399, 2000.
\newblock With an appendix by Gerhard Niklasch.

\bibitem{MOT21}
M.~J. Mossinghoff, T.~Oliveira~e Silva, and T.~S. Trudgian.
\newblock The distribution of {$k$}-free numbers.
\newblock {\em Math. Comp.}, 90(328):907--929, 2021.

\bibitem{MT17}
M.~J. Mossinghoff and T.~S. Trudgian.
\newblock The {L}iouville function and the {R}iemann hypothesis.
\newblock In {\em Exploring the {R}iemann zeta function}, pages 201--221. Springer, Cham, 2017.

\bibitem{OtR85}
A.~M. Odlyzko and H.~J.~J. te~Riele.
\newblock Disproof of the {M}ertens conjecture.
\newblock {\em J. Reine Angew. Math.}, 357:138--160, 1985.

\bibitem{P05}
F.~Pappalardi.
\newblock A survey on {$k$}-freeness.
\newblock In {\em Number theory}, volume~1 of {\em Ramanujan Math. Soc. Lect. Notes Ser.}, pages 71--88. Ramanujan Math. Soc., Mysore, 2005.

\bibitem{1919.Polya}
G.~P\'olya.
\newblock {Verschiedene Bemerkungen zur Zahlentheorie}.
\newblock {\em Jahresbericht der deutschen Math.--Vereinigung}, 28:31--40, 1919.

\bibitem{SMC06}
J.~S\'andor, D.~S. Mitrinovi\'c, and B.~Crstici.
\newblock {\em Handbook of number theory. {I}}.
\newblock Springer, Dordrecht, 2006.
\newblock Second printing of the 1996 original.

\bibitem{S09}
K.~Soundararajan.
\newblock Partial sums of the {M}\"obius function.
\newblock {\em J. Reine Angew. Math.}, 631:141--152, 2009.

\bibitem{S77}
D.~Suryanarayana.
\newblock On a theorem of {A}postol concerning {M}\"{o}bius functions of order {$k$}.
\newblock {\em Pacific J. Math.}, 68(1):277--281, 1977.

\bibitem{T80}
M.~Tanaka.
\newblock On the {M}\"{o}bius and allied functions.
\newblock {\em Tokyo J. Math.}, 3(2):215--218, 1980.

\bibitem{T15}
G.~Tenenbaum.
\newblock {\em Introduction to analytic and probabilistic number theory}, volume 163 of {\em Graduate Studies in Mathematics}.
\newblock American Mathematical Society, Providence, RI, third edition, 2015.

\bibitem{W63}
A.~Walfisz.
\newblock {\em Weylsche {E}xponentialsummen in der neueren {Z}ahlentheorie}, volume~XV of {\em Mathematische Forschungsberichte}.
\newblock VEB Deutscher Verlag der Wissenschaften, Berlin, 1963.

\end{thebibliography}

\end{document}